\documentclass[11pt]{amsart}
\usepackage[margin=1in]{geometry}
\usepackage{amsmath}
\usepackage{amssymb}
\usepackage{amsthm}
\usepackage{mathabx}
\usepackage{graphicx}
\usepackage{algorithm}
\usepackage{algorithmic}
\usepackage{mathtools}
\usepackage{float}
\usepackage[table]{xcolor}
\usepackage[bottom]{footmisc}

\theoremstyle{plain}
\newtheorem{theorem}{Theorem}[section]
\newtheorem{lemma}[theorem]{Lemma}
\newtheorem{corollary}[theorem]{Corollary}
\newtheorem{proposition}[theorem]{Proposition}

\theoremstyle{definition}
\newtheorem{definition}[theorem]{Definition}
\newtheorem{example}[theorem]{Example}
\newtheorem{conjecture}[theorem]{Conjecture}

\newtheorem{question}[theorem]{Question}
\newtheorem{remark}[theorem]{Remark}

\numberwithin{equation}{section}

\newcommand{\calA}{\mathcal{A}} 
\newcommand{\calB}{\mathcal{B}} 
\newcommand{\calE}{\mathcal{E}} 
\newcommand{\calH}{\mathcal{H}} 
\newcommand{\calM}{\mathcal{M}}
\newcommand{\calN}{\mathcal{N}}
\newcommand{\calO}{\mathcal{O}}
\newcommand{\calR}{\mathcal{R}} 
\newcommand{\calT}{\mathcal{T}}
\newcommand{\bbP}{\mathbb{P}}
\newcommand{\bbR}{\mathbb{R}}
\newcommand{\bbS}{\mathbb{S}}
\newcommand{\bbZ}{\mathbb{Z}}

\newcommand{\bfa}{\mathbf{a}}
\newcommand{\bfb}{\mathbf{b}}
\newcommand{\bfc}{\mathbf{c}}
\newcommand{\bfm}{\mathbf{m}}
\newcommand{\bfp}{\mathbf{p}}
\newcommand{\bfr}{\mathbf{r}}
\newcommand{\bfv}{\mathbf{v}}
\newcommand{\bfz}{\mathbf{z}}
\newcommand{\bfM}{\mathbf{M}}
\newcommand{\row}{\rightarrow}
\renewcommand{\phi}{\varphi}
\renewcommand{\hat}{\widehat}
\DeclareMathOperator{\lcm}{lcm}

\title[A billiards-like dynamical system for attacking chess pieces]{A billiards-like dynamical system \\ for attacking chess pieces}

\author{Christopher R.\ H.\ Hanusa}
\address{Department of Mathematics \\ Queens College (CUNY) \\ 65-30 Kissena Blvd. \\ Queens, NY 11367-1597, U.S.A.}
\email{\tt chanusa@qc.cuny.edu}

\author{Arvind V. Mahankali}
\address{Carnegie Mellon University, 5032 Forbes Avenue, SMC 6925, Pittsburgh, PA 15289}
\email{\tt amahanka@andrew.cmu.edu}

\begin{document}

\begin{abstract}
We apply a one-dimensional discrete dynamical system originally considered by Arnol'd reminiscent of mathematical billiards to the study of two-move riders, a type of fairy chess piece.  In this model, particles travel through a bounded convex region along line segments of one of two fixed slopes.  

We apply this dynamical system to characterize the vertices of the inside-out polytope arising from counting placements of nonattacking chess pieces and also to give a bound for the period of the counting quasipolynomial.  The analysis focuses on points of the region that are on trajectories that contain a corner or on cycles of full rank, or are crossing points thereof.  

As a consequence, we give a simple proof that the period of the bishops' counting quasipolynomial is 2, and provide formulas bounding periods of counting quasipolynomials for many two-move riders including all partial nightriders.  We draw parallels  
to the theory of mathematical billiards and pose many new open questions.
\end{abstract}

\subjclass[2010]{Primary 05A15, 37C83, 37E15; Secondary 00A08, 52C05, 52C35.}


\keywords{Nonattacking, chess pieces, fairy chess, riders, Ehrhart theory, inside-out polytope, quasipolynomial, trajectories, discrete dynamical system, billiards, Poincar\'e map}

\maketitle


\section{Introduction} 

The classic $n$-Queens Problem asks in how many ways $n$ nonattacking queens can be placed on an $n\times n$ chessboard.  In a series of six papers \cite{qqs1,qqs2,qqs3,qqs4,qqs5,qqs6}, Chaiken, Hanusa, and Zaslavsky develop a geometric approach involving lattice point counting to answer a generalization when the board is made up of integer lattice points on the interior of an $n$-dilation of a convex polygon $\calB$, pieces $\bbP$ are riders (which means they can travel arbitrarily far in a move's direction like a queen, bishop, or the fairy nightrider), and the number of pieces $q$ is decoupled from the size of the board.  Their main structural result (Theorem~4.1 of \cite{qqs1}) is that the number of nonattacking configurations of $q$ $\bbP$-pieces on the $(n+1)$-dilation of $\calB^\circ$ is always a quasipolynomial in $n$ of degree $2q$.

In this paper we investigate the period of this counting quasipolynomial when the pieces have exactly two moves, on any board and for any number of pieces.  (Pieces with only one move are completely understood while pieces with three or more moves are much more complex, as discussed in \cite{qqs4}.) We learn that this period is determined by the behavior of an extension of a one-dimensional discrete dynamical system originally considered by Arnol'd (described in \cite{2006}).  This discrete dynamical system is similar to that of mathematical billiard theory in that particles travel across a region along line segments and ``bounce'' when they hit the region's boundary. However, instead of obeying the law of reflection, the line segments have one of two slopes determined by the moves of the fairy chess piece. Compare the diagrams in Figure~\ref{fig:comparison}.  

\begin{figure}[htp]
\centering
\raisebox{.05in}{\includegraphics[width=5.5cm]{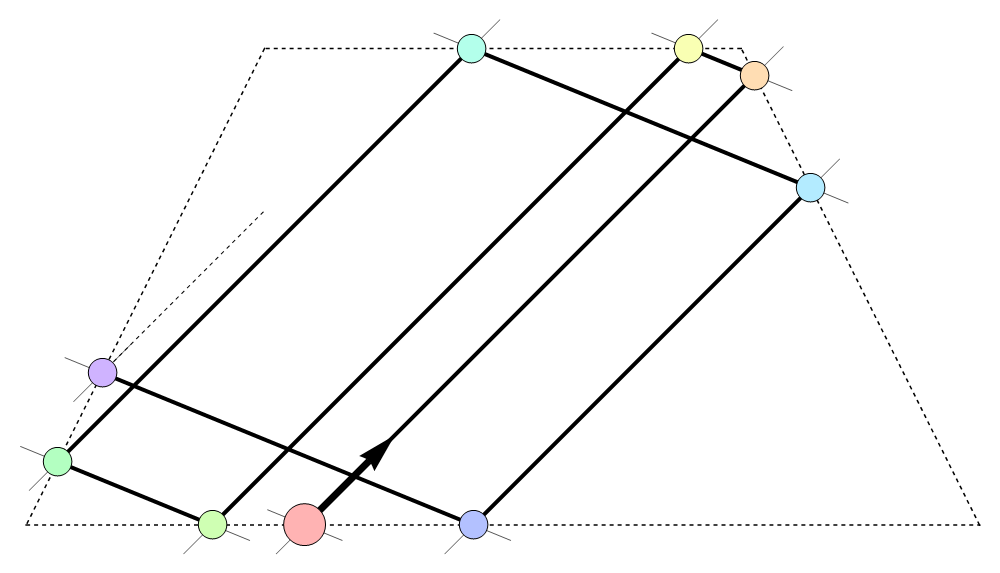}}
\includegraphics[width=5.7cm]{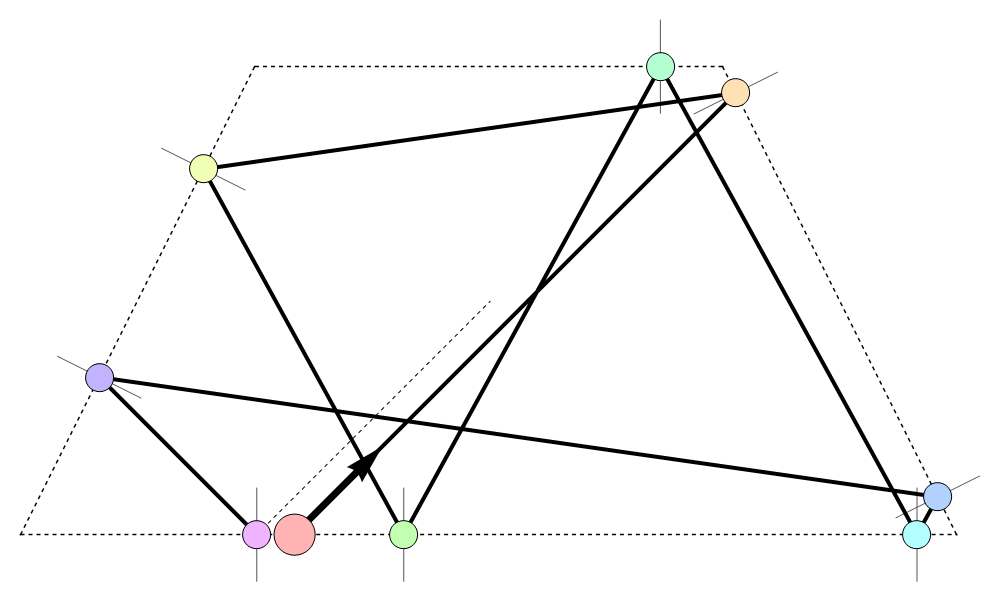}
\caption{A comparison of the behavior of two discrete dynamical systems in a convex region.  On the left is the discrete dynamical system where the particle bounces off a wall in directions that alternate between the moves of a fairy chess piece.
On the right is the classical discrete dynamical system from mathematical billiards in which the particle bounces off a wall by obeying the law of reflection.}
\label{fig:comparison}
\end{figure}

The study of mathematical billiards has been a fruitful area of research for over a hundred years; some early papers were written by Artin \cite{Artin} and Birkhoff \cite{Birkhoff}.  The work of Sina\u{i} \cite{Sinai} stimulated interest in the ergodic theory and chaos of billiards, and the connections to geometry, statistical physics, and Teichm\"uller theory give billiards a wide appeal.  We recommend the surveys by Tabachnikov, Masur, and Gutkin \cite{tabachnikov, chapter13, gutkin1, gutkin2}.  
		
The theory of the dynamical system studied in this article originally appears in \cite{GKT} and \cite{KT}, where it has been developed for ovals (smooth convex closed curves) to study geodesics on Lorentz surfaces.  It also appears when considering light-like trajectories within ellipses in the Minkowski plane \cite{DR}. We discuss this dynamical system on all bounded convex regions, in turn developing theory that we apply to convex polygons.  This leads to a number of open questions motivated by our study and by the billiards literature. For example, the particle flows can be periodic, can converge to a limit set, or exhibit ergodicity, and it is not clear when each property occurs. (See Sections~\ref{sec:periodic} and \ref{sec:properties}.) 

Counting lattice points in polytopes is the subject of a field of mathematics named Ehrhart theory after the work of Eug\`ene Ehrhart \cite{Ehrhart}. Ehrhart theory has found applications in integer programming, number theory, and algebra, among others \cite{DeLoera, BBKSZ, flow}; for more background, see the accessible works by De Loera \cite{DeLoera} and Beck and Robins \cite{Beck}. Beck and Zaslavsky \cite{IOP} count lattice points in a polytope that avoid an arrangement of hyperplanes.  Such a construction is called an {\em inside-out polytope}; it is under this framework that the $q$-Queens Problem was converted into a counting question.  Ehrhart theory tells us that the period of the counting quasipolynomial always divides the denominator of the inside-out polytope---the least common multiple of the denominators of its vertices.

Theorem~\ref{thm:vertices} characterizes the vertices of the inside-out polytope for two-move riders as points on flows (trajectories) in the dynamical system.   Vertices either involve trajectory segments that include corners of the board or cyclic trajectory segments whose system of defining equations is linearly independent (rigid cycles) or interior crossing points of these trajectory segments.  This characterization allows us to prove a formula for the denominator of the counting quasipolynomial for the number of nonattacking chess piece configurations in Theorem~\ref{thm:denominator}.  From a dynamical systems point of view, this theorem is interesting because of the necessity of explicitly calculating the crossing points of flows, which are rarely considered in a dynamical systems context. (We are aware only of Don's \cite{Don}.) Investigating crossing points in the context of billiard theory may lead to further insights in that field.  

When we analyzed the trajectories to calculate bounds on periods of the counting quasipolynomials we saw some striking behavior.  Section~\ref{sec:case1} highlights a case where there are no rigid cycles and the corner trajectories are well behaved. Section~\ref{sec:oppsigns} discusses a case where there is one rigid cycle that serves as an attractor to all other trajectories. In Section~\ref{sec:case3}, the dynamical system reduces to that of billiards. In Section~\ref{sec:periodic} we show an example where the trajectories appear to behave ergodically.  These insights allow us to provide insight into a question of periodicity discussed by Khmelev \cite{khmelev_2005}.  We are able to show that the measure of the set of pairs of slopes that lead to a periodic orbit or a fixed point in a convex polygon is strictly positive.  We show that it is of full measure for a triangle and conjecture that it is not of full measure for any other convex polygon.

One of the motivations of this work was to better understand nightriders, riders that move like the knight along slopes of $\pm2$ and $\pm\frac{1}{2}$, whose behavior was investigated in \cite{qqs5}.  The authors suggested that partial nightriders---two-move riders with a subset of the nightrider's moves---would be fruitful pieces to investigate.  Indeed, in Section~\ref{sec:applications} we are able to determine denominators (and therefore bounds on the period of the counting quasipolynomial) of all two-move partial nightriders.  Our work also gives a simple new proof that the period of the counting quasipolynomial for $q\geq3$ bishops is 2, avoiding the need to use signed graph theory which was present in the original proof given in \cite{qqs6}.

We now share a brief summary of our paper.   We recall the necessary background information from the theory of chess piece configurations in Section~\ref{sec:background} and we explore hyperplanes and rank in Section~\ref{sec:hyperplanes}.  Section~\ref{sec:dynamical} defines the discrete dynamical system and concepts related to trajectories.  In Section~\ref{sec:proofs}, we apply the dynamical system to polygonal boards which allows us to characterize vertices of the inside-out polytope in Theorem~\ref{thm:vertices} and prove the formula for its denominator in Theorem~\ref{thm:denominator}.  We then restrict to the square board to find explicit formulas for the coordinates of points on trajectories and crossing points in Section~\ref{sec:applications}, culminating with the discussion of periodic trajectories in Section~\ref{sec:periodic}. We conclude with a wide variety of open problems in Section~\ref{sec:open}, asking questions about future regions of study, properties of trajectories, and generalizations of the dynamical system, among others.

\section{Background}\label{sec:background}
We gather here the necessary Ehrhart and nonattacking chess piece theory background information and notation from \cite{qqs1,qqs2,qqs4}.  Every $q$-Queens Problem involves three parameters, a board $\calB$, a piece $\bbP$, and a positive integral number of pieces~$q$.  

Our {\em board} $\calB$ is a convex polygon whose corners have rational coordinates; we use the notation $\calB^\circ$ and $\partial\calB$ for its interior and boundary, respectively.  (This is not to be confused with rational polygons, defined in billiard theory whose angles are rational multiples of $\pi$.)  These boards are dilated by an integer factor of $(n+1)$; pieces are placed on integer lattice points in $(n+1)\calB^\circ\cap\bbZ^2$. The {\em square board} refers to $\calB=[0,1]^2$. 

A {\em piece} $\bbP$ has a set $\bfM$ of non-parallel {\em basic moves} $\bfm=(c,d)$ where $c$ and $d$ are relatively prime integers; a piece at position $(x,y)$ may move to any position $(x,y)+k\bfm$ for $k\in\bbZ$ and $\bfm\in\bfM$.  (This ability to move arbitrarily far along a basic move is the defining property of a {\em rider}.)  For example, the {\em bishop} is the piece with basic moves $(1,1)$ and $(-1,1)$, while the fairy chess {\em nightrider} is the rider with the basic moves $(1,\pm2)$ and $(2,\pm1)$ of the knight.  

In this article we consider pieces that are {\em two-move riders} with basic moves $\bfm_1=(c_1, d_1)$ and $\bfm_2=(c_2, d_2)$.  Three pieces that were proposed in \cite{qqs5} and which motivated our study are the {\em partial nightriders}: the {\em lateral nightrider} moves along lines of slope $\pm1/2$, the {\em inclined nightrider} moves along lines of slope $1/2$ and $2$, and the {\em orthonightrider} moves along lines of slope $1/2$ and $-2$.

Two pieces are said to {\em attack} if their positions differ by a multiple of a move. A {\em configuration} of $q$ pieces corresponds to an integral point $\bfz=(\bfz_1,\ldots,\bfz_q)\in\left((n+1)\calB\right)^q\subseteq\bbR^{2q}$ and is said to be {\em nonattacking} if no two pieces are attacking.  Mathematically, a configuration is nonattacking if it avoids the {\em hyperplane arrangement $\calA_\bbP^q$} consisting of all {\em attack equations of type $r$},
\begin{equation}\label{eq:attackeq}
    (\bfz_i-\bfz_j)\cdot (d_r,-c_r)=0,
\end{equation}
for $1\leq i <  j\leq q$ and $r=1,2$; we adopt the shorthand notation $\bfz_i \sim_r \bfz_j$ for Equation~\eqref{eq:attackeq}. Note that $\sim_r$ is an equivalence relation.  

This construction from \cite{qqs1} converts the question of counting the number of nonattacking configurations of $q$ $\bbP$-pieces on $(n+1)\calB^\circ$, denoted $u_\bbP(q;n)$, into a lattice point counting question in this {\em inside-out polytope}, denoted $(\calB^q,\calA_\bbP^q)$.   The boundary equations of $\calB$ are avoided as well, which justifies counting configurations in $((n+1)\calB)^q\cap\bbZ^{2q}$ instead of $((n+1)\calB^\circ)^q\cap\bbZ^{2q}$.

A {\em vertex} of $(\calB^q,\calA_\bbP^q)$ is any point of $\calB^q$ that is the intersection of attack equations from $\calA_\bbP^q$ and {\em fixation equations} (or simply {\em fixations}) of the form 
\begin{equation}
    (\alpha_1,\alpha_2)\cdot \bfz_i=\beta,
\end{equation}
where $\alpha_1 x+\alpha_2 y = \beta$ is the equation of a side of $\mathcal{B}$.  The {\em denominator} $\Delta(\bfz)$ of a vertex $\bfz$ is the least common multiple of the denominators of its coordinates, and the {\em denominator} $D(\calB^q,\calA_\bbP^q)$ of an inside-out polytope is the least common multiple of the denominators of all its vertices. In Theorem~\ref{thm:vertices} we determine the structure of all vertices of the inside-out polytope for an arbitrary board $\calB$ and a two-move rider $\bbP$. 

As with many counting questions in Ehrhart Theory, the main structural result of \cite{qqs1} is that $u_\bbP(q;n)$ is always a quasipolynomial in $n$ of degree $2q$.  That is, for each fixed $q$, $u_\bbP(q;n)$ is given by a cyclically repeating sequence of polynomials in $n$ and its {\em period} $p$ is the shortest length of such a cycle.  The period of the counting quasipolynomial $u_\bbP(q;n)$ always divides the denominator $D(\calB^q,\calA_\bbP^q)$ \cite[Theorem~3.23]{Beck}.   In Ehrhart Theory the period is often difficult to obtain and can be much smaller than this denominator, but in chess counting problems the period and denominator always seem to agree which leads to the following conjecture.

\begin{conjecture}[{{\cite[Conjecture~8.6]{qqs2}}}]
\label{conj:period}
The period of the counting quasipolynomial $u_\bbP(q;n)$ equals the denominator $D([0,1]^{2q},\calA_\bbP^q)$.
\end{conjecture}

\section{Hyperplanes and Rank.}\label{sec:hyperplanes}

We define the following concepts related to the geometry of the inside-out polytope.

\begin{definition}
For $\bfz = (\bfz_1, \bfz_2, \ldots, \bfz_k) \in \calB^{q}$ we define $\calH(\bfz)$, the \emph{hyperplane arrangement associated to $\bfz$}, to be the set of all attack equations and fixations on which $\bfz$ lies.
\end{definition} 

In other words, $\calH(\bfz)$ will include the attack equation $\bfz_i\sim_r \bfz_j$ if pieces $i$ and $j$ attack and will include the fixation $(\alpha_1,\alpha_2)\cdot\bfz_i = \beta$ if and only if $\bfz_i$ lies on the edge of $\calB$ defined by $\alpha_1 x +\alpha_2 y = \beta$. 

The rank of hyperplane arrangements, equations, and sets of points will help determine when $\bfz \in \calB^q$ is a vertex of $(\calB^q, \calA_\bbP^q)$.

\begin{definition}
The \emph{rank} of a hyperplane arrangement $\calH$ in $\bbR^d$ is the rank of the system of equations given by its hyperplanes. $\calH$ has \emph{full rank} if it has rank $d$.
We say the \emph{rank} of a point $\bfz \in \bbR^{2q}$ is the rank of $\calH(\bfz)$, and $\bfz$ has \emph{full rank} if $\calH(\bfz)$ has full rank.
We say the \emph{rank} of a set $S = \{\bfz_1, \ldots, \bfz_k\} \subseteq \bbR^2$ is the rank of the point $\bfz = (\bfz_1, \ldots, \bfz_k)$, and $S$ has \emph{full rank} if $\bfz$ has full rank.
\end{definition}

\begin{definition}
A set $\calH$ of hyperplanes in $\bbR^d$ is said to be \emph{linearly independent} if the rank of $\calH$ is equal to its size, or equivalently, if the set of normal vectors to these hyperplanes is linearly independent.
\end{definition}

\begin{lemma}\label{lem:fullrankvertex}
$\bfz \in \calB^{q}$ has full rank if and only if $\bfz$ is a vertex of $(\calB^q, \calA_{\bbP}^q)$.
\end{lemma}
\begin{proof}
Suppose $\bfz$ (and therefore $\calH(\bfz)$) has full rank. By removing redundant hyperplanes, $\calH(\bfz)$ can be reduced to a linearly independent set of hyperplanes $\calH$ of full rank of which $\bfz$ is the intersection point, so  $\bfz$ is a vertex of $(\calB^q, \calA_{\bbP}^q)$.  If $\bfz$ is a vertex, $\calH(\bfz)$ contains this $\calH$, so $\calH(\bfz)$ (and therefore $\bfz$) has full rank.
\end{proof}

\begin{example}\label{ex:fullrank}
Consider the inclined nightrider on the square board with moves $\bfm_1=(2,1)$ and $\bfm_2=(1,2)$.  

When $\bfz = (0, 0, 1, 1/2)$, $\calH(\bfz)$ contains the fixations
$x_1 = 0$, $y_1 = 0$, and $x_2 = 1$ and the attack equation $\bfz_1 \sim_1 \bfz_2$. These four equations form a system of full rank; we conclude  $\calH(\bfz)$ and $\bfz$ have full rank and $\bfz$ is a vertex of  $([0,1]^4,\calA_\bbP^2)$.

When 
$\bfz = (0, 0, 0, 0, 1, 1)$,
$\calH(\bfz)$ consists of the fixations $x_1 = 0$, $y_1 = 0$, $x_2 = 0$, $y_2 = 0$, $x_3 = 1$, and $y_3 = 1$ and the attack equations $\bfz_1 \sim_1 \bfz_2$ and $\bfz_1 \sim_2 \bfz_2$  since $\bfz_1 = \bfz_2$.  $\calH(\bfz)$ contains eight equations; the attack equations are redundant because the fixations uniquely determine $\bfz$; those six equations form a system of full rank, so $\calH(\bfz)$ and $\bfz$ have full rank, and $\bfz$ is a vertex of $([0,1]^6,\calA_\bbP^3)$.

When $\bfz = (1, 1/2, 3/4, 0)$, $\calH(\bfz)=\{
x_1 = 1, y_2 = 0,\bfz_1 \sim_2 \bfz_2$\}, which has rank at most $3$, so $\calH(\bfz)$ is not of full rank and $\bfz$ is not a vertex of $([0,1]^4,\calA_\bbP^2)$.
\end{example}

\begin{lemma} \label{lem:twoequationprop}
Suppose $\calH$ is a hyperplane arrangement consisting of hyperplanes in $\bbR^{2k}$, and \[\bfz = (x_1, y_1, x_2, y_2, \ldots, x_k, y_k) \in \bbR^{2k}\] is the unique intersection point of the elements of $\calH$. Then, for all $i$ between $1$ and $k$, $\calH$ contains at least $2$ hyperplanes whose equations involve either $x_i$ or $y_i$.
\end{lemma}
\begin{proof}
Suppose there exists $i$ between $1$ and $k$ such that $\calH$ contains at most one hyperplane with equation involving $(x_i, y_i)$.  Take $\calH$ to contain $2k$ linearly independent hyperplanes (removing redundant hyperplanes as necessary). 

Since $\calH$ only contains one equation involving $x_i$ or $y_i$, $\calH$ contains at least $2k - 1$ hyperplanes whose equations only involve the other $2k-2$ variables, which contradicts the linear independence of $\calH$.
\end{proof}

The rank of a point depends only on the set of its constituent coordinate pairs:
 
\begin{proposition}\label{prop:duplicatepoints}
$\bfz = (\bfz_1, \bfz_2, \ldots, \bfz_q) \in \calB^q$ has full rank if and only if $\bfz' = (\bfz_1, \ldots, \bfz_q, \bfz_q) \in \calB^{q + 1}$ has full rank.
\end{proposition}
\begin{proof}
First, suppose $\bfz$ has rank $2q$. Then there is a hyperplane arrangement $\calH$ with rank $2q$, whose members are attack equations and fixations involving $\bfz_1, \ldots, \bfz_q$ and whose set $\calN$ of normal vectors forms a basis of $\bbR^{2q}$. Therefore the hyperplane arrangement 
$$\calH \cup \{\bfz_{q + 1} \sim_1 \bfz_q, \bfz_{q + 1} \sim_2 \bfz_q\}$$
is also linearly independent because the set 
$$\calN \cup \{(0, 0, \ldots, d_1, -c_1, -d_1, c_1),(0, 0, \ldots, d_2, -c_2, -d_2, c_2)\}$$
forms a basis of $\bbR^{2(q + 1)}$. 

Now, suppose $\bfz'$ has full rank, so that it is the unique intersection point of a linearly independent hyperplane arrangement $\calH'$, consisting of $2q+2$ attack equations and fixations.  Without loss of generality, we can assume $\calH'$ contains the hyperplanes
$$\bfz_{q + 1} \sim_1 \bfz_q \textup{ and } \bfz_{q + 1} \sim_2 \bfz_q$$
If not, we can add these to $\calH'$ and remove two redundant hyperplanes. 

We can ensure that $\calH'$ has at most two attack equations involving $\bfz_{q+1}$ and no fixations involving $\bfz_{q+1}$ by replacing all other occurrences of $\bfz_{q+1}$ by $\bfz_q$. Then, this equivalent system of equations has exactly two equations involving $\bfz_{q+1}$; removing these two equations leaves $2q$ linearly independent equations involving $\bfz_1$ through $\bfz_q$, so $\bfz$ has rank $2q$.
\end{proof}

The following observation is straightforward but helpful to state explicitly.

\begin{lemma}\label{lem:multiplesolutions}
Let $\bfz=(\bfz_1,\ldots,\bfz_q)\in\calB^q$.  If there exists a point $\bfz'=(\bfz_1',\ldots,\bfz_q')\in\calB^q$ such that $\calH(\bfz)=\calH(\bfz')$ and the sets $\{\bfz_i\}_{1\leq i\leq q}$ and $\{\bfz_i'\}_{1\leq i\leq q}$ are different, then $\bfz$ is not of full rank.
\end{lemma}
\begin{proof}
Because there are two points $\bfz,\bfz'\in\bbR^{2q}$ that satisfy the same system of equations,  $\calH(\bfz)$ (and therefore $\bfz$) is not of full rank.
\end{proof}

\section{The discrete dynamical system for fairy chess}
\label{sec:dynamical}

In this section we make precise the discrete dynamical system that arises naturally in our study of attacking chess piece configurations. The first appearance of such a discrete dynamical system appears to be in works of Arnol'd from the 1950's (see \cite{2006}).  Previous 
study including \cite{GKT, KT} has focused on ellipses and other smooth closed curves and Khmelev studied finitely many 
break-type singularities \cite{khmelev_2005}.  Here we develop the theory further to apply to all bounded convex regions.  Our presentation has been informed by surveys on the billiard model by Gutkin \cite{gutkin2} and Tabachnikov \cite{tabachnikov}. Open problems related to this system have been gathered in Section~\ref{sec:open}.

We start with any bounded convex region $\calR$ (our board) and any nonparallel pair of vectors $\bfm_1$ and $\bfm_2$ (our basic moves). We let $\calM\subseteq \mathbb{S}^1$ consist of the four unit vectors parallel to $\bfm_1$ or $\bfm_2$.  We investigate the movement of a particle, determined by its position $\bfr\in\calR$ and its velocity $\bfv$, restricted to be an element of $\calM$.  The particle moves along the ray starting at $\bfr$ in the direction $\bfv$ until it hits a point $\bfb$ on the boundary of $\calR$, denoted $\partial\calR$.  

In this discrete dynamical system, the particle ``bounces'' differently from billiards. The convexity of $\calR$ implies $\bfb$ has at most two vectors from $\calM$ pointing toward the interior of $\calR$, including $-\bfv$. When there is a second vector $\bfv'$ pointing toward the interior of $\calR$, the particle ``bounces'' and leaves $\bfb$ in that direction, as exemplified in Figure~\ref{fig:dynam}.   
When there is no second vector---which can occur at a corner of $\calR$ or at a point of tangency of $\bfm_1$ or $\bfm_2$, as in Figure~\ref{fig:extended}(c)---we adopt the convention that the particle pauses imperceptibly at $\bfb$ and then returns along $-\bfv$. Going backward in time is as simple as applying the same dynamics after negating the velocity vector.  As such, the particle meanders through $\calR$ on lines parallel to $\bfm_1$ and $\bfm_2$.
(In previous work of Genin, Khesin, and Tabachnikov \cite{GKT} and Khesin and Tabachnikov \cite{KT}, these line segments are geodesics on Lorentz surfaces and are called null lines.)

\begin{figure}[tbp]
    \centering
    \includegraphics[height=2.1in]{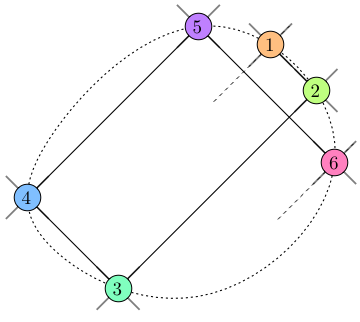}
    \caption{With basic moves $(1,1)$ and $(1,-1)$, consecutive boundary points along the flow lie on lines of slope $1$ and $-1$.}
    \label{fig:dynam}
\end{figure}

Formally, the {\em phase space} $\Psi$ is the quotient of the set $$\{(\bfr, \bfv) \mid \bfr\in \calR, \bfv\in\calM\}$$
by the identifications $(\bfb,\bfv)=(\bfb,\bfv')$ for $\bfb \in \partial\calR$ and nonparallel $\bfv, \bfv' \in \calM$ when $\bfv$ points away from the interior of $\calR$ and $\bfv'$ points toward the interior of $\calR$, as well as $(\bfb,\bfm_1)=(\bfb,\bfm_2)$ and $(\bfb,-\bfm_1)=(\bfb,-\bfm_2)$ for $\bfb \in \partial\calR$ if at most one $\bfv\in\calM$ points toward the interior of $\calR$.  (We have identified only two elements of the set (in an arbitrary manner) instead of all four so that $\bfb$ is repeated in the Poincar\'e map below.) 
The {\em flow} $F^t : \Psi \row \Psi$ of the particle is how the pair $(\bfr, \bfv)$ changes over time: when $\bfr$ is in the interior of $\calR$, it moves with velocity $\bfv$, while once it reaches $\partial\calR$, it switches velocity to $\bfv'$. If a particle reaches $\partial\calR$ with velocity $\bfv$ and neither $\bfv'$ nor $-\bfv'$ points toward the interior of $\calR$, the particle pauses imperceptibly and switches velocity to $-\bfv$. If no vector $\bfv\in\calM$ points toward the interior of $\calR$ from $\bfb\in\partial\calR$, the particle remains at $\bfb$ indefinitely.

The {\em Poincar\'e section} $\Phi = \{(\bfb, \bfv) \in \Psi \mid \bfb \in \partial\calR\}$ is the restriction of the phase space to points in the boundary of $\calR$ and the \textit{chess attack map} $\phi: \Phi \rightarrow \Phi$ is the {\em Poincar\'e map} which describes the transition from one boundary point to the next.  (This chess attack map is the concept analogous to the billiard map.  Further, in \cite{GKT}, their circle map $T$ is equivalent to our $\phi^2$.)

The orbit of a flow $F^t$ yields a doubly-infinite sequence $[(\bfb_i,\bfv_i)]_{i\in\bbZ}$ where $\phi(\bfb_i,\bfv_i)=(\bfb_{i+1},\bfv_{i+1})$ and $\phi(\bfb_i,-\bfv_i)=(\bfb_{i-1},-\bfv_{i-1})$.  
When we record only the points $[\bfb_i]_{i\in\bbZ}$ of this sequence we will call this a {\em trajectory} and again use $\phi$ to denote the transition $\phi(\bfb_i)=\bfb_{i+1}$ when the velocity vector is understood.  We use square brackets for trajectories to differentiate them from ordered $n$-tuples of points in $\calR$.

We say that a trajectory is {\em periodic} if there exists an integer $p\geq 1$ such that $\phi^{n+p}(\bfb)=\phi^{n}(\bfb)$ for all integers $n$, and define its {\em period} to be the smallest such $p$. There are three types of periodic trajectories that appear in this dynamical system: fixed point trajectories, reflection-symmetric periodic trajectories, and cyclic trajectories. 

When a trajectory $T=[\bfb]_{i\in\bbZ}$ has period $p=1$, we say that $\bfb$ is a {\em fixed point} and $T$ is a {\em fixed point trajectory}. Fixed points cannot occur when $\calR$ is a smooth curve; the only place they occur is at a corner of $\calR$ when no vector $\bfv\in\calM$ points toward the interior of $\calR$. 

The period of a periodic trajectory that is not a fixed point trajectory must always be even because the slopes of the incident vectors alternate between being parallel to $\bfm_1$ and $\bfm_2$. 

When a periodic trajectory $T$ with period $p$ satisfies $\bfb_i=\bfb_{i+1}$ for some $i\in\bbZ$, then the trajectory exhibits reflection symmetry in that \[[\bfb_{i+1},\hdots,\bfb_{i+p/2}]=[\bfb_{i+p},\hdots,\bfb_{i+p/2+1}].\]
The flow continually bounces back and forth between the two path endpoints located at $\bfb_{i+1}$ and $\bfb_{i+p/2}$, at which there is only one vector of $\calM$ that points toward the interior of $\calR$ due to a move vector tangency or an arrival at a corner with no second viable direction. We call this a {\em reflection-symmetric periodic trajectory}.

The last type of periodic trajectory is a {\em cyclic trajectory}, in which $[\bfb_1,\hdots,\bfb_p]$ consists of $p$ distinct vectors. Each point $\bfb_i$ of a cyclic trajectory has two vectors of $\calM$ that point toward the interior of $\calR$ from $\bfb_i$. 

Given a point $\bfb\in\partial\calR$, define its {\em trajectory set} to be the set of points in the trajectory $[\bfb_i]_{i\in\bbZ}$ starting at $\bfb$. With this definition, the trajectory set of $\bfb$ is finite if and only if the trajectory through $\bfb$ is periodic. 

\begin{example}\label{ex:extended}
Figure~\ref{fig:extended} exhibits three trajectories.  In Figure~\ref{fig:extended}(a), the dynamical system corresponds to the square board and the basic moves $(10,3)$ and $(11,8)$.  This trajectory is not periodic, nor are there any periodic trajectories other than the fixed points at the upper left and bottom right corners, as proved in Proposition~\ref{prop:samesignconverge}.  

Figure~\ref{fig:extended}(b) shows a hexagonal board with basic moves $(1,2)$ and $(2,1)$.  The chosen trajectory is periodic and overlaps itself infinitely many times; its six points make up a trajectory set.  

The non-polygonal board in Figure~\ref{fig:extended}(c) is made up of two circular curves and one line segment.  The basic moves are $(1,1)$ and $(0,1)$.  We show an example of a periodic trajectory in the corresponding dynamical system---the board has a vertical tangent at $\bfb_2$ and the point $\bfb_{-1}$ is located at a corner with no points of the board accessible vertically. The associated reflection-symmetric periodic trajectory with period 8 is 
\[[\hdots,\bfb_{0},\bfb_{1},\bfb_{2},\bfb_{3},\bfb_{4},\bfb_{5},\bfb_{6},\bfb_{7},\bfb_{8},\bfb_{8},\hdots]=[\hdots,
{\color{blue!20!green!70!black}\bfb_{0}},
{\color{purple!10!blue!70!black}\bfb_{1}},
{\color{magenta!90!black}\bfb_{2}},
{\color{magenta!90!black}\bfb_{2}},
{\color{purple!10!blue!70!black}\bfb_{1}},
{\color{blue!20!green!70!black}\bfb_{0}},
{\color{yellow!20!orange!80!black}\bfb_{-1}},
{\color{yellow!20!orange!80!black}\bfb_{-1}},
{\color{blue!20!green!70!black}\bfb_{0}},
{\color{purple!10!blue!70!black}\bfb_{1}},
\hdots].\]
\end{example}

\begin{figure}[tbp]
\makebox[.2cm]{\raisebox{4.5cm}{(a)}}
\includegraphics[width=5cm]{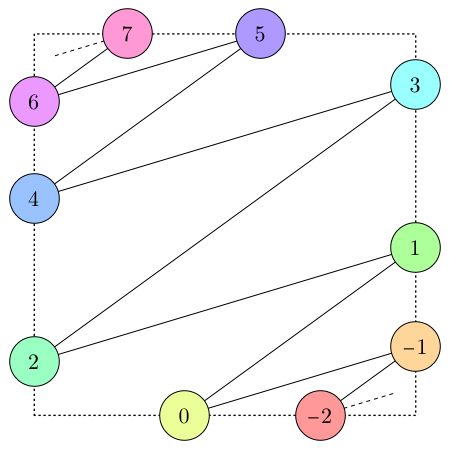}\quad
\makebox[0.2cm]{\raisebox{4.5cm}{(b)}}
\raisebox{.25cm}{\includegraphics[width=5cm]{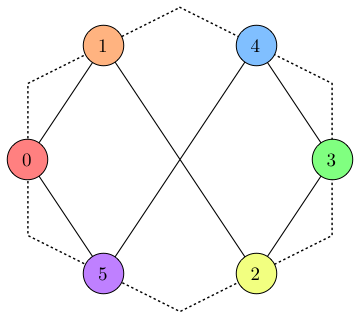}} 
\makebox[0cm]{\raisebox{4.5cm}{(c)}}
\raisebox{.7cm}{\includegraphics[width=5cm]{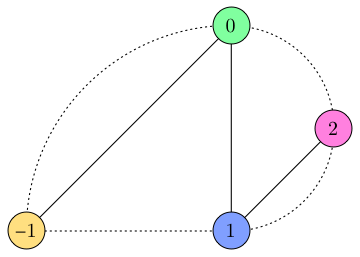}} 
\caption{The behavior of three trajectories for the dynamical systems discussed in Example~\ref{ex:extended}.}
\label{fig:extended}
\end{figure}

It will be useful to also describe the chess attack map using the following {\em antipode maps}, which are involutions on $\partial \calR$, and originate from the case of one-move riders in \cite{qqs4}.

\begin{definition} \label{def:nextpointfunction}
For a bounded convex region $\calR$ and a pair of vectors $\bfm_1=(c_1, d_1)$ and $\bfm_2=(c_2, d_2)$, define $s_r: \partial \calR \rightarrow \partial \calR$ for $r = 1, 2$ as follows. Suppose $\bfb \in \partial \calR$, and consider the line
$$\ell = \{\bfb + \lambda(c_r, d_r)\, |\, \lambda \in \bbR\}.$$
If $\ell \cap \calB^\circ=\emptyset$, define $s_r\bfb = \bfb$. Otherwise, since $\calR$ is convex, $\ell \cap \partial\calR$ has exactly 2 elements and we define $s_r\bfb$ to be the other element.
\end{definition}

The chess attack map for a point $\bfb\in\partial\calR$ and a velocity $\bfv$ pointing toward the interior of $\calR$ can then be described as 
\(\phi(\bfb)=s_r\bfb\), where $\bfv$ is parallel to $\bfm_r$.

For a point $\bfb\in\partial\calR$ and a direction $\bfv\in\calM$ pointing toward the interior of $\calR$, we define a {\em trajectory segment} to be a finite sequence $T=[\bfb,\phi(\bfb),\ldots,\phi^{l-1}(\bfb)]$ of {\em distinct} points.  (The reader should note that in this definition our restriction to distinct points is nonstandard.) We say $T$ has {\em length} $l$.  Equivalently, a trajectory segment is a consecutive subsequence of a trajectory with distinct points.  

Note that when $T$ is part of a cyclic trajectory of period $p$, then the longest trajectory segment $[\bfb_1,\bfb_2,\ldots,\bfb_l]$ is of length $p$ and satisfies $\phi(\bfb_p)=\bfb_1$. We call such a $T$ a {\em cyclic trajectory segment}; it necessarily contains all points in the trajectory set of $\bfb_1$.  The trajectory segment $[\bfb_0,\bfb_1,\hdots,\bfb_5]$ from Figure~\ref{fig:extended}(b) is a cyclic trajectory segment. 

We see that any trajectory segment in $\calR$ can be obtained by alternately applying $s_1$ and $s_2$ to an initial point $\bfb$. In other words, every trajectory segment is of the form
$$[\bfb, s_1\bfb, s_2s_1\bfb, s_1s_2s_1\bfb, \ldots] \quad \textup{ or } \quad [\bfb, s_2\bfb, s_1s_2\bfb, s_2s_1s_2\bfb, \ldots].$$

Critical to our study of periods of counting quasipolynomials are both the points on trajectory segments $T = [\bfb_1, \ldots,\bfb_l]$ and points on the interior of $\calR$ where flows that extend a bit on either side of $\bfb_1$ and $\bfb_l$ cross.  

\begin{definition}
\label{def:crossingpoint}
Let $T_a = [\bfa_1, \bfa_2, \ldots,\bfa_k]$ and $T_b = [\bfb_1, \bfb_2, \ldots,\bfb_l]$ be trajectory segments in $\calR$. We say $\bfc$ is a \textbf{crossing point} of $T_a$ and $T_b$ if $\bfc \in \calR^{\circ}$ and there exist some $i$ and $j$ such that $\bfc$ is contained in the line segments from $\bfa_i$ to $\bfa_{i+1}$ and from $\bfb_j$ to $\bfb_{j+1}$ for some $1\leq i\leq k-1$ and $1\leq j\leq l-1$. If $T_a=T_b$, we say $\bfc$ is a \textbf{self-crossing} point of $T_a$. See Figure~\ref{fig:crossing}.
\end{definition}

\begin{definition}
\label{def:augmentation}
Let $T = [\bfb_1, \ldots,\bfb_l]$ be a trajectory segment in $\calR$. Then $T$ is a consecutive subsequence of a trajectory $T'=[\ldots,\bfb_1, \ldots,\bfb_l,\ldots]$.  We define the {\em augmentation} of $T$ to be the sequence of points including $\bfb_1$ through $\bfb_l$ where we prepend $\bfb_0$ from $T'$ and we postpend $\bfb_{l+1}$.  
\end{definition}

\begin{remark}
An augmentation of a cyclic trajectory segment will no longer qualify as a trajectory segment because it has repeated vertices. On the other hand, the flow corresponding to the augmentation of a cyclic trajectory segment $T$ traces out the entire cycle that the trajectory traverses.  Furthermore, crossing points of augmentations of trajectory segments may exist that are not crossing points of the trajectory segments themselves, as shown in Figure~\ref{fig:crossing}(a). Last, in the case of a trajectory segment where $\bfb_0=\bfb_1$ or $\bfb_l=\bfb_{l+1}$, its augmentation is also not a trajectory segment, but no additional line segments (nor crossings) have been created. 
\end{remark}

\begin{figure}[t]
\makebox[.2cm]{\raisebox{4.5cm}{(a)}}
\includegraphics[width=5cm]{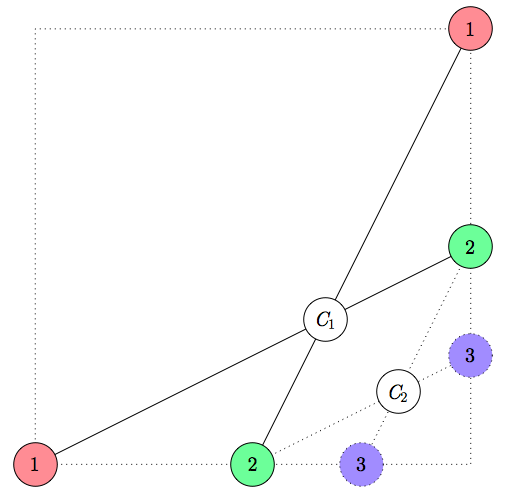} \qquad\qquad
\makebox[.2cm]{\raisebox{4.5cm}{(b)}}
\includegraphics[width=5cm]{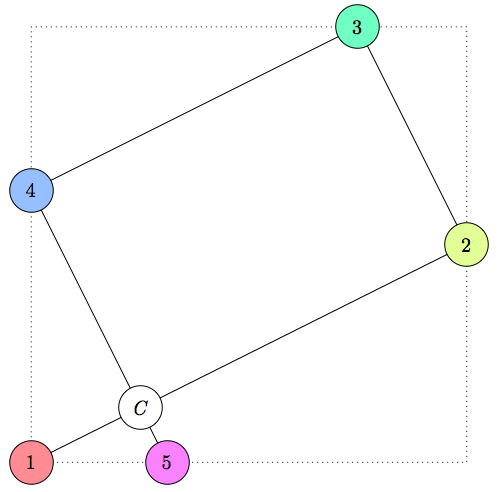}
\caption{(a) For the square board when $\bbP$ has moves $(2, 1)$ and $(1, 2)$, the two two-point trajectory segments starting at $(0, 0)$ and $(1, 1)$ have a crossing point at $C_1=(2/3,1/3)$.  The augmentations of these trajectory segments (for which we have only shown the forward continuation to the points labeled `3') have a crossing point at $C_2=(5/6,1/6)$.  (b) When $\bbP$ has moves $(2, 1)$ and $(1, -2)$, the five-point corner trajectory segment starting at $(0,0)$ has a self-crossing point at $(1/4,1/8)$.
}
\label{fig:crossing}
\end{figure}

\section{Trajectories on polygonal boards} \label{sec:proofs}

We apply the discrete dynamical system to the $q$-Queens Problem by restricting to general convex polygonal regions $\calB$.  We prove a characterization of the set of vertices $\bfz=(\bfz_1,\dots,\bfz_q)$ of the inside-out polytope $(\calB^q, \calA_\bbP^q)$ that depends on whether the points $\bfz_i$ lie on certain trajectory segments or are crossing points thereof. 

\subsection{Corner trajectory segments and rigid cycles}  It is natural to extend the notion of rank to a trajectory segment $T$ in $\calB$.  We define the {\em rank} of a trajectory segment $T$ to be the rank of the collection of points in $T$ (recall that the points of $T$ must all be distinct).  We characterize the types of trajectory segments that are of full rank. 

\begin{definition} \label{def:cornertrajectory}
A trajectory segment $T$ is called a \emph{corner trajectory segment} if it contains a corner of $\calB$.  
\end{definition}

\begin{definition} \label{def:cyclicaltrajectory}
Let $T=[\bfb_1, \ldots, \bfb_k]$ be a cyclic trajectory segment.  If the point $(\bfb_1, \ldots, \bfb_k)$ has full rank, $T$ is called a \emph{rigid cycle}; otherwise $T$ is called a \emph{treachery}.
\end{definition}

Only for certain choices of $\calB$ and $\bbP$ do rigid cycles exist.  The characterization of when they exist is open; see Question~\ref{q:rigid}.

\begin{example}\label{ex:rigidcycle}
Let $\calB = [0, 1]^2$ and consider the piece $\bbP$ with moves $\bfm_1=(m, 1)$ and $\bfm_2=(-1, m)$ where $m > 1$.  Choose $\bfb_1=(x_1,y_1)$ along the south edge of $\calB$, so that $\bfb_2=(x_2,y_2)=s_1\bfb_1$ lies along its east edge, $\bfb_3=(x_3,y_3)=s_2\bfb_2$ lies along its north edge, and $\bfb_4=(x_4,y_4)=s_1\bfb_3$ lies along its west edge.  If $\bfb_1=s_2\bfb_4$, the trajectory segment $T=[\bfb_1,\bfb_2,\bfb_3,\bfb_4]$ is cyclic and the coordinates of the points are given by the system of equations
\begin{equation}\label{eq:system}
\{\bfb_1\sim_1\bfb_2,~\bfb_2\sim_2\bfb_3,~\bfb_3\sim_1\bfb_4,~\bfb_4\sim_2\bfb_1,~y_1 = 0,~x_2 = 1,~y_3 = 1,~x_4 = 0\}.
\end{equation}
$T$ is a rigid cycle because when $m>1$ the unique solution to this system is 
$$\bfz = \Big(\frac{1}{1 + m}, 0, 1, \frac{1}{1 + m}, \frac{m}{1 + m}, 1, 0, \frac{m}{1 + m}\Big).$$
Notice this implies $\bfz$ is a vertex of $([0,1]^8, \calA_\bbP^4)$. Figure~\ref{fig:periodic}(a) shows the special case when $m = 2$.  This example is generalized and studied in Section~\ref{sec:oppsigns}.
\end{example}

\begin{figure}[tbp]
\makebox[.2cm]{\raisebox{4.5cm}{(a)}}
\includegraphics[width=5cm]{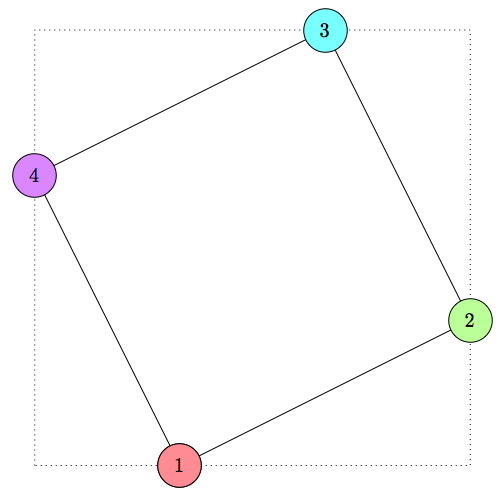}\qquad\qquad
\makebox[.2cm]{\raisebox{4.5cm}{(b)}}
\includegraphics[width=5cm]{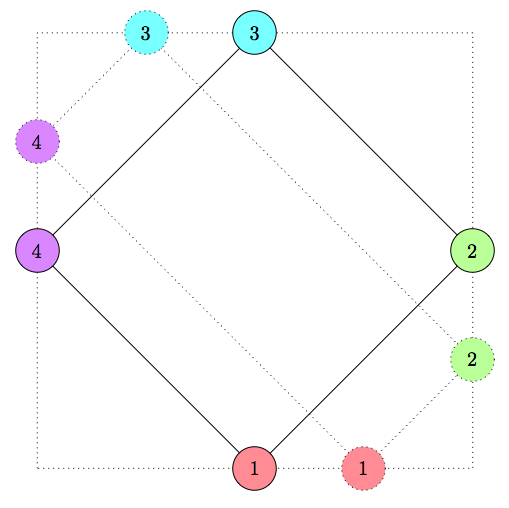}
\caption{(a) For the piece with basic moves $(2, 1)$ and $(1, -2)$ the cyclic trajectory segment starting at $\bfb_1 = (\frac{1}{3}, 0)$ is a rigid cycle.  See Example~\ref{ex:rigidcycle}.  (b) For the bishop, every cyclic trajectory segment is a treachery.  Note that the solid and dotted trajectory segments have the same associated hyperplane arrangement $\calH(\bfb)$. See Example~\ref{ex:bishop}.}
\label{fig:periodic}
\end{figure}

\begin{example}\label{ex:bishop}
When $\calB = [0, 1]^2$ and $\bbP$ is the bishop with moves $(1, 1)$ and $(1, -1)$, there are no rigid cycles. Trajectory segments fall into two cases---either they contain two opposite corners of $\calB$ or they form a cyclic trajectory segment $T=[(x,0),(1,1-x),(1-x,1),(0,x)]$ which is a rectangle. The points of $T$ have as their associated hyperplane arrangement the system \eqref{eq:system} when $m=1$, which is no longer of full rank. We conclude $T$ is a treachery.  Alternatively, we see that all cyclic trajectory segments satisfy the system \eqref{eq:system}, so by Lemma~\ref{lem:multiplesolutions}, they are not of full rank.  See Figure~\ref{fig:periodic}(b).
\end{example}

Figure~\ref{fig:periodic}(b) suggests that treacheries can be displaced slightly to obtain new treacheries.  We will make this precise in Theorem~\ref{thm:treacheryinterval}. The following lemma is essential for classifying trajectory segments that correspond to vertices of the inside-out polytope.

\begin{lemma} \label{lemma:cornertrajectoryrank}
Corner trajectory segments have full rank.
\end{lemma}
\begin{proof}
We show that every corner trajectory segment $T=[\bfb_1,\ldots,\bfb_{l}]$ has full rank by induction on $l$. When $l=1$, $\bfb_1$ is a corner and hence the intersection of two linearly independent fixations; we conclude $T$ has rank~$2$.

Now suppose $l > 1$ is an integer, and all corner trajectory segments of shorter length $l'<l$ have full rank. Suppose that $\bfb_l$ is not a corner of $\calB$, so that  $T'=[\bfb_1,\ldots,\bfb_{l-1}]$ remains a corner trajectory segment, and therefore has full rank. Then $\bfz'=(\bfb_1,\ldots,\bfb_{l-1})$ is the unique intersection point of a set $\calH'$ of $2l-2$ hyperplanes.

The point $\bfb_l$ equals $s_r\bfb_{l-1}$ for some $r\in\{1,2\}$ and also lies along an edge $\alpha_1 x + \alpha_2 y=\beta$ of $\calB$.  The set of equations $\calE=\{(\alpha_1,\alpha_2)\cdot\bfz_l=\beta, \bfb_{l-1}\sim_r\bfz_l\}$ is linearly independent because $\bfb_l - \bfb_{l - 1}$ is not parallel to the edge $\alpha_1 x + \alpha_2 y=\beta$. (Had the difference been parallel, the definition of $s_r$ would imply that $\bfb_{l}=s_r\bfb_{l-1}$, however there are no repeated vertices in a trajectory segment.)

Therefore the set of $2l$ hyperplanes $\calH=\calH' \cup \calE$ is linearly independent and uniquely defines the vertex $\bfz=(\bfb_1,\ldots,\bfb_{l})$; we conclude that $T$ has full rank.

Finally, if $\bfb_l$ is a corner of $\calB$, we can use a similar argument by removing $\bfb_1$.
\end{proof}

\begin{proposition} \label{prop:cornertrajectoryrank}
The only trajectory segments of full rank are corner trajectory segments and rigid cycles. 
\end{proposition}
\begin{proof}
We show that non-cyclic trajectory segments $T$ that are of full rank must contain a corner. The statement then follows from Lemma~\ref{lemma:cornertrajectoryrank}.

Suppose $T = [\bfb_1, \ldots, \bfb_l]$ has rank $2l$ and is not cyclic. Let $\bfz = (\bfb_1, \ldots, \bfb_l) \in (\partial\calB)^l$. There exists a set of hyperplanes $\calH \subseteq \calH(\bfz)$ with size and rank $2l$ whose unique intersection point is $\bfz$. Since $T$ is not cyclic, it either does not contain one of $s_1\bfb_l$ and $s_2\bfb_l$, or one of these is equal to $\bfb_l$ (if the line parallel to a move does not intersect the interior of $\calB$). 

In both cases, there can only be one attack equation between $\bfb_l$ and another point $\bfb_j$. Since the points of $T$ are distinct, for all $j$ between $2$ and $l - 1$, $\bfb_j$ can only be related to $\bfb_{j - 1}$ and $\bfb_{j + 1}$ through attack equations. Finally, $\bfb_1$ can only be related to $\bfb_2$ through an attack equation. In summary, $\calH$ contains at most $l - 1$ attack hyperplanes. If each $\bfb_j$ lies on only one fixation, then $\calH$ contains at most $2l - 1$ hyperplanes, which is impossible because $\calH$ has full rank. Therefore, at least one of the $\bfb_j$ is a corner of $\calB$.
\end{proof}

We now prove a characterization of when a cyclic trajectory segment is a treachery.

\begin{theorem} \label{thm:treacheryinterval}
Let $\calB$ be a convex polygon and let $T$ be a cyclic trajectory segment that is not a corner trajectory segment. Then $T = [\bfb_1, \ldots, \bfb_k]$ is a treachery if and only if there exists a $\delta>0$ and a family of cyclic trajectory segments
$T(t)=[\bfb_1(t), \ldots, \bfb_k(t)]$ such that 
$T(0)=T$, $\bfb_i(t)$ is a continuous function of $t$, and $\bfb_i(t)$ lies on the interior of the same side of $\calB$ as $\bfb_i$ for all $t \in (-\delta, \delta)$. 
\end{theorem}
\begin{proof}
Let $\bfb=(\bfb_1, \ldots, \bfb_k)$, $\calH$ be the hyperplane arrangement associated to $\bfb$, and  $A$ be the $|\calH|\times 2k$ matrix given by the equations of the hyperplanes of $\calH$ in the standard bases.  Since $T$ is a treachery there is a non-zero vector $\bfv=(\bfv_1, \bfv_2, \ldots, \bfv_k)\in \bbR^{2k}$ in the null space of $A$. We will show that there exists a $\delta>0$ such that for all $t\in(-\delta,\delta)$,  \[T(t)=[\bfb_1+t\bfv_1,\ldots,\bfb_k+t\bfv_k]\] is a treachery in $\calB$ with points lying on the same sides of $\calB$ as $T$.

The hyperplanes of $\calH$ include the fixations that correspond to the edges $E_i$ of $\calB$ on which the points $\bfb_i$ lie.  That is, $(\alpha_1,\alpha_2)\cdot \bfb_i=\beta$ for the appropriate $\alpha_1$, $\alpha_2$, and $\beta$.  As a consequence, $(\alpha_1,\alpha_2)\cdot \bfv_i=0$ so $\bfv_i$ must be parallel to the edge $E_i$.  Therefore the set of points defined by $I_i=\bfb_i+t\bfv_i$ for $t$ in an interval $(-\delta_i,\delta_i)$ lie along the line that includes the edge $E_i$.  There may be other points of $T$ that lie on this line, so choose $\delta_i>0$ such that it contains no points of $E_i$ that are more than halfway to any other point of $T$ or any vertex of $\calB$.  This is possible because the points of $T$ are distinct and a treachery contains no corner points by Lemma~\ref{lemma:cornertrajectoryrank}.
Choose $\delta=\min(\delta_1,\ldots,\delta_k)$.  Then $T(t)$ is a set of trajectory segments with points lying on the same sides of $\calB$ as $T$ for all $t\in(-\delta,\delta)$.

Since $A(\bfb+t\bfv) = A\bfb$ for  $t\in(-\delta,\delta)$, $\bfb+t\bfv$ satisfies the same set of attack equations and fixations as $\bfb$.  (which are not of full rank). The above choice of $\delta$ ensures no other attack equations or fixations are satisfied by the points of $T(t)$. This is because no point of $T(t)$ becomes a corner point and no two points of $T(t)$ coalesce.  Therefore $T(t)$ is a treachery for all $t\in(-\delta,\delta)$.

\medskip
To prove the converse, suppose there exists a $\delta > 0$ such that $T = [\bfb_1, \ldots, \bfb_k]$ is a member of a family of cyclic trajectory segments $T(t) = [\bfb_1(t), \ldots, \bfb_k(t)]$ for $t \in (-\delta, \delta)$ in which $T(0)=T$, $\bfb_i(t)$ is a continuous function of $t$, and $\bfb_i(t)$ lies on the interior of $E_i$ for all $i$. 

We will show that $T$ is a treachery by showing that some of the members of the family $T(t)$ satisfy the same attack equations and fixations as $T$ does, thereby showing that the set of these equations is not of full rank. 

First, every member of $T(t)$ satisfies the same fixations because the points $\bfb_i(t)$ lie on the same edges. 
Moreover, since $\bfb_i$ and $\bfb_{i+1}$ are attacking along move vector $\bfm_r$ for $r=1$ or $r=2$, we know that there exists an $\epsilon_i>0$ such that for all $t\in(-\epsilon_i,\epsilon_i)$, $\bfb_i(t)$ and $\bfb_{i+1}(t)$ must also attack along $\bfm_r$ because $\bfm_1$ and $\bfm_2$ have fixed non-equal slopes.  We can also ensure that no new attack equations appear by reducing $\epsilon_i$ further if necessary, similar to the argument above. By taking $\epsilon=\min(\delta,\epsilon_1,\hdots,\epsilon_k)$, we ensure that for all $t\in(-\epsilon,\epsilon)$, $T(t)$ satisfies the exact same set of attack equations as $T$. We conclude that $T$ is a treachery.
\end{proof}

\begin{remark}
No two treacheries in the continuous family $T(t)$ share a point. This is because neither move vector $\bfm_1$ nor $\bfm_2$ is parallel to any edge $E_i$ since the $T(t)$'s are cyclic trajectory segments.
\end{remark}

\subsection{The vertices and denominator of the inside-out polytope}

We will determine the vertices and denominator of the inside-out polytope by understanding points $(\bfb_1,\ldots,\bfb_q)\in\calB^q$. 

\begin{lemma} \label{lemma:partitionintotrajectories}
Let $\calR$ be a bounded convex region, $\bbP$ be a piece with basic moves $\bfm_1$ and $\bfm_2$, and $S$ be a finite subset of $\partial\calR$.  Then $S$ can be partitioned into a set of trajectory segments $\calT(S)$ that travel along paths parallel to $\bfm_1$ and $\bfm_2$. 
\end{lemma}
\begin{proof}
Suppose $S = \{\bfz_1, \ldots, \bfz_q\}$.  Create a graph with vertices labeled by $S$ with an edge $\{\bfz_i, \bfz_j\}$ if $\bfz_i \neq \bfz_j$ and $\bfz_i = s_1\bfz_j$ or $\bfz_i = s_2\bfz_j$. Every vertex in this graph has degree at most $2$, so each connected component is either a cycle or a path. Within each connected component the edges will alternate between corresponding to $s_1$ and $s_2$. Writing the vertices of a connected component in the order given by the path or cycle gives a trajectory segment. 
\end{proof}

\begin{lemma} \label{lemma:pointpartition}
Let $\calB$ be a bounded convex polygon, $\bbP$ be a piece with basic moves $\bfm_1$ and $\bfm_2$, and $S$ be a finite subset of $\partial\calB$.  The rank of $S$ is the sum of the ranks of the trajectory segments in $\calT(S)$; $S$ has full rank if and only if each of the trajectory segments in $\calT(S)$ has full rank.
\end{lemma}
\begin{proof}
Let $S=\{\bfz_1,\ldots,\bfz_l\}$ and define $\bfz=(\bfz_1,\ldots,\bfz_l)$.  The rank of $S$ equals the rank of $\calH(\bfz)$, which includes all hyperplanes from each individual trajectory segment in $\calT(S)$ and attack equations between distinct trajectory segments in $\calT(S)$.  These latter equations do not exist unless $\bfz_i$ and $\bfz_j$ are in different connected components and lie on the same side of $\calB$ that is parallel to a move of $\bbP$.  In this case there are two fixations in $\calH(\bfz)$, with equations involving $\bfz_i$ and $\bfz_j$ respectively, whose equations imply the attack equation linking $\bfz_i$ and $\bfz_j$. This means we can remove the attack equation with this equation from $\calH(\bfz)$ without affecting the rank of $\calH(\bfz)$.  This concludes the proof.
\end{proof}

This tells us exactly which $(\bfb_1,\ldots,\bfb_q)\in(\partial\calB)^q$ have full rank. We can extend this knowledge to points in $\calB^q$.

\begin{theorem}\label{thm:vertices}
Suppose $\bfz = (\bfz_1, \bfz_2, \ldots, \bfz_q) \in \calB^q$ and partition  $S=\{\bfz_i\}_{1 \leq i \leq q}$ into $B\subseteq\partial\calB$ and $C\subseteq\calB^\circ$.
Then $\bfz$ is a vertex of $(\calB^q, \calA_{\bbP}^q)$ if and only if:
\begin{enumerate}
    \item $B$ can be written as the union of corner trajectory segments and rigid cycles, and 
    \item $C$ consists of crossing points of augmentations of these corner trajectory segments and rigid cycles (which may include self-crossing points).
\end{enumerate}
\end{theorem}
\begin{proof}
By Lemma~\ref{lem:fullrankvertex} and Proposition~\ref{prop:duplicatepoints}, $\bfz$ is a vertex of $(\calB^q, \calA_{\bbP}^q)$ if and only if $S$ has full rank. We proceed by induction on the size of $C$. When $|C|=0$, $S=B$; Proposition~\ref{prop:cornertrajectoryrank} and Lemma~\ref{lemma:pointpartition} show that $S$ has full rank if and only if $S$ can be decomposed into corner trajectory segments and rigid cycles.

Now let $|C|=k>0$.  Suppose $\bfz$ has full rank and let $\calH\subseteq\calH(\bfz)$ be a set of $2q$ equations whose unique intersection point is $\bfz$.  Up to index reordering, we can choose $\bfz_q \in \calB^\circ$ and therefore $\bfz_q$ is not involved in any fixations. Furthermore we can assume $\calH$ contains exactly one attack equation involving $\bfz_q$ of each type, say $\bfz_q \sim_1 \bfz_i$ and $\bfz_q \sim_2 \bfz_j$ for $1\leq i,j\leq q-1$. (If there were more than one, we could replace an equation of the form $\bfz_q\sim_1\bfz_k$ by $\bfz_i\sim_1\bfz_k\in\calH(\bfz)$.) The removal of these two equations from $\calH$ gives $2q-2$ linearly independent equations involving $\bfz_1$ through $\bfz_{q-1}$, so $\bfz'=(\bfz_1,\ldots,\bfz_{q-1})$ is a vertex of $\big(\calB^{q-1}, \calA_{\bbP}^{q-1}\big)$.  

By induction, the set $S'=\{\bfz_i\}_{1 \leq i \leq q-1}$ can be partitioned into the sets $B'=B\subseteq\partial\calB$ and $C'\subseteq\calB^\circ$ which satisfy conditions (1) and (2). 
Every $\bfz_k\in C'$ is the crossing point of trajectory segments involving points of $B$, so is related by attacking equations of both types to points of $B$. Since $\bfz_q \sim_1 \bfz_i$ and $\bfz_q \sim_2 \bfz_j$, then by transitivity of $\sim_r$, $\bfz_q$ is a crossing point of augmentations of trajectory segments involving points of $B$.  (The need for augmentations of trajectory segments $T=[\bfb_1,\ldots,\bfb_l]$ arises because the crossing point may lie along the line segment leaving $\bfb_1$ toward $\bfb_0$ or along the line segment leaving $\bfb_{l}$ toward $\bfb_{l+1}$.) This completes the proof in the forward direction.

Now, suppose the elements of $C$ are all crossing points of augmentations of the corner trajectory segments and rigid cycles of $\calT(B)$. By the inductive hypothesis, $\bfz' = (\bfz_1, \ldots, \bfz_{q - 1})$ has full rank. Let $\calH' \subseteq \calH(\bfz)$ be a set of hyperplanes with rank $2(q - 1)$, whose intersection is $\bfz'$. 

Since $\bfz_q$ is a crossing point of two of the augmentations of trajectory segments making up $B$, it is linked by two attack equations of different types to points in $B$. Since the moves of $\bbP$ are linearly independent, $\calH'$ with these two attack equations appended has rank $2q$, and $\bfz$ is the intersection point of these hyperplanes. Therefore, $\bfz$ has full rank.
\end{proof}

Now that we know the vertices of $(\calB^q, \calA_{\bbP}^q)$, we can find its denominator.

\begin{theorem}\label{thm:denominator}
The denominator of $(\calB^q, \calA_{\bbP}^q)$ is equal to the least common multiple of the denominators of
\begin{enumerate}
    \item Points on rigid cycles of length at most $q$,
    \item Points on corner trajectory segments of length at most $q$ that start at corners,
    \item Self-crossing points of augmentations of corner trajectory segments or rigid cycles of length at most $q-1$, and 
    \item Crossing points of augmentations of two distinct corner trajectory segments or rigid cycles whose lengths sum to at most $q - 1$.
\end{enumerate}
\end{theorem}
\begin{proof}
The denominator of $(\calB^q, \calA_{\bbP}^q)$ is the least common multiple of the denominator of all vertices $\bfz = (\bfz_1, \bfz_2, \ldots, \bfz_q)$ of $(\calB^q, \calA_{\bbP}^q)$.  We must determine the set of all points that may occur as a component of some vertex.  

Theorem~\ref{thm:vertices} says that the set of points $S=\{\bfz_i\}$ can be partitioned into corner trajectory segments, rigid cycles, and crossing points of augmentations of these trajectory segments.  We first consider points on corner trajectory segments and rigid cycles.  Points on rigid cycles $[\bfb_1,\ldots,\bfb_l]$ of length $l\leq q$ will occur as components of the vertex $(\bfb_1,\ldots,\bfb_l,\bfb_l,\ldots,\bfb_l)\in(\partial\calB)^q$.  Points on corner trajectory segments $[\bfb_1,\ldots,\bfb_l]$ that include the corner $\bfc$ will occur as components of a vertex $(\bfb_1',\ldots,\bfb_q')\in(\partial\calB)^q$ where $T=(\bfb_1',\ldots,\bfb_l')$ is a trajectory segment starting at $\bfb_1'=\bfc$ and continues until $l=q$ or until it stops.  (If $l<q$, we pad our vertex with repeated points $\bfb_{l+1}'=\cdots=\bfb_q'=\bfc$.) 

A point $\bfc$ that occurs as a self-crossing point of an augmentation of some trajectory segment $T=[\bfb_1,\ldots,\bfb_l]$ occurs as a vertex $(\bfb_1,\ldots,\bfb_l,\bfc,\ldots,\bfc)$ if and only if $l\leq q-1$, and a point $\bfc$ that occurs as a crossing point of augmentations of trajectory segments $T_a=[\bfa_1,\ldots,\bfa_k]$ and $T_b=[\bfb_1,\ldots,\bfb_l]$ occurs as a vertex $(\bfa_1,\ldots,\bfa_k,\bfb_1,\ldots,\bfb_l,\bfc,\ldots,\bfc)$ if and only if $k+l\leq q-1$.
\end{proof}

\begin{corollary}
If Conjecture~\ref{conj:period} is true, the period of the counting quasipolynomial $u_\bbP(q;n)$ on the square board is equal to the least common multiple of the denominators of
\begin{enumerate}
    \item Points on rigid cycles of length at most $q$,
    \item Points on corner trajectory segments of length at most $q$ that start at corners,
    \item Self-crossing points of augmentations of corner trajectory segments or rigid cycles of length at most $q-1$, and 
    \item Crossing points of augmentations of two distinct corner trajectory segments or rigid cycles whose lengths sum to at most $q - 1$.
\end{enumerate}
\end{corollary}

Theorem~\ref{thm:denominator} allows us to give a new and simpler proof of the main result from \cite{qqs6}.

\begin{corollary}
For $q\geq 3$, the period of the counting quasipolynomial of the bishop on the square board is $2$.
\end{corollary}

\begin{proof}
The only corner trajectory segments are the diagonals of $\calB$, and there are no rigid cycles, as shown in Example~\ref{ex:bishop}.  This shows that every vertex $\bfz$ of $(\calB^q, \calA_{\bbP}^q)$ has $z_i$ equal to a corner of $\calB$ or $(\frac{1}{2},\frac{1}{2})$.  Therefore the denominator of the inside-out polytope is $2$, which the period of the counting quasipolynomial must divide  \cite[Theorem~4.1]{IOP}. Lemma~3.3(III) from \cite{qqs3} shows that the coefficient of $n^{2q-6}$ has period 2 for $q\geq 3$, which completes the proof.
\end{proof}

\section{Two-move riders on square boards}\label{sec:applications}

We now restrict to the square board $\calB=[0,1]^2$ and investigate the denominator $D([0, 1]^{2q}, \calA_{\bbP}^q)$ of the inside-out polytope for some two-move riders.  Our analysis is broken into cases depending on the signs and magnitudes of the slopes ${d_1}/{c_1}$ and ${d_2}/{c_2}$.  We will notate the open edges of $\calB$ counterclockwise by 
\[\textup{$E_1 = (0, 1) \times \{0\}$, $E_2 = \{1\} \times (0, 1)$, $E_3 = (0, 1) \times \{1\}$, and $E_4 = \{0\} \times (0, 1)$.}\]

\subsection{Slopes of the same sign}\label{sec:case1} 

First consider a piece whose moves have slopes of the same sign.  The non-trivial trajectories converge to the fixed points of the dynamical system.  This proposition does not require the slopes to be rational.

\begin{proposition}
\label{prop:samesignconverge}
Let $\calB$ be the square board and let $\bbP$ have moves with real-valued slopes $m_1$ and $m_2$ of the same sign.  When the slopes are positive, the only periodic trajectories involve the the fixed points $(0,1)$ and $(1,0)$; whereas when the slopes are negative, the only periodic trajectories involve the fixed points $(0,0)$ and $(1,1)$. Every other trajectory $T=[\bfb_n]_{n\in\bbZ}$ is not periodic, with its points converging to the corresponding diametrically opposed fixed points as $n$ approaches $+\infty$ or $-\infty$.
\end{proposition}

\begin{proof}
Assume $0<m_1<m_2$.  The points $(0,1)$ and $(1,0)$ are fixed points of the system because the lines of slope $m_1$ and $m_2$ starting there do not intersect the interior of the square.   We show the trajectory set of every other point of $\partial\calB$ has infinite cardinality.
Define the sets $$Z_1 = E_1 \cup E_4 \cup \{(0, 0)\}
\textup{ and }
Z_2 = E_2 \cup E_3 \cup \{(1, 1)\}.$$
When $\bfb\in  Z_1$, both $s_1\bfb$ and $s_2\bfb$ are in $Z_2$ and $s_1\bfb$ is to the southeast of $s_2\bfb$; when $\bfb\in Z_2$, both $s_1\bfb$ and $s_2\bfb$ are in $Z_1$ and $s_1\bfb$ is to the northwest of $s_2\bfb$.  

Therefore, when $\bfb\in Z_1$, then $s_2\bfb\in Z_2$, so 
$s_1s_2\bfb$ is to the northwest of $s_2s_2\bfb=\bfb$ in $Z_1$, and hence $s_1s_2\bfb$ is closer to $(0, 1)$ than $\bfb$ is. This is the first of the following statements, all of which follow similarly.
\begin{gather*}
    \textup{When $\bfb \in Z_1$, 
    $0<|(0, 1) - s_1s_2\bfb| < |(0, 1) - \bfb|$  and 
    $0<|(1, 0) - s_2s_1\bfb| < |(1, 0) - \bfb|.$}\\  
    \textup{When $\bfb \in Z_2$, 
    $0<|(1, 0) - s_1s_2\bfb| < |(1, 0) - \bfb|$ and 
    $0<|(0, 1) - s_2s_1\bfb| < |(0, 1) - \bfb|$.}
\end{gather*}

We conclude that the trajectory $T$ is not periodic with one tail continuing northwest and one tail continuing southeast; we now show its points converge to $(1,0)$ or $(0,1)$.  When successive points alternate between neighboring sides, the distance to $(1,0)$ or $(0,1)$ along the same edge decreases geometrically.  The trajectory may first alternate between diametrically opposite sides, but in that case, the distance between consecutive points along the same edge is a positive constant, so the trajectory eventually begins to alternate between neighboring sides.

The negative slope case follows by symmetry.
\end{proof}

We now apply Theorem~\ref{thm:denominator} to find $D([0,1]^{2q}, \calA_{\bbP}^q)$ when $0 < m_1 < 1 < m_2$.  This restriction avoids a much more complicated formula that arises from the behavior of the crossing points in the general case.

\begin{theorem} \label{thm:inclineddenominator}
Suppose $\bbP$ has moves $\bfm_1=(c_1, d_1)$ and $\bfm_2=(c_2, d_2)$, satisfying $0 < \frac{d_1}{c_1} < 1 < \frac{d_2}{c_2}$. 
The denominator of $([0,1]^{2q}, \calA_{\bbP}^q)$ is the least common multiple of the denominators of the first $q$ terms of the following sequence defined for $i\geq 1$
\begin{equation}\label{eq:denominclined1}
    \begin{cases}
    (1,\big(\frac{d_1c_2}{c_1d_2}\big)^{\frac{i - 1}{2}}) & \textup{for $i$ odd}\\
    \Big(\frac{d_1}{c_1}\big(\frac{d_1c_2}{c_1d_2}\big)^{\frac{i}{2} - 1},\frac{c_2}{d_2}\big(\frac{d_1c_2}{c_1d_2}\big)^{\frac{i}{2} - 1}\Big) & \textup{for $i$ even}\\
\end{cases}
\end{equation}
and the denominators of the first $\lfloor {(q - 1)}/{2} \rfloor$ terms of the following sequence defined for $i\geq 1$
\begin{equation}\label{eq:denominclined2}
\begin{cases}
\big(\frac{d_1c_2}{c_1d_2}\big)^{\frac{i - 1}{2}}\left(\frac{c_2(d_1 - c_1)}{c_1d_2 - c_2d_1}, \frac{d_1(d_2 - c_2)}{c_1d_2 - c_2d_1}\right) &  \textup{for $i$ odd}\\
\big(\frac{d_1c_2}{c_1d_2}\big)^{\frac{i}{2}}\left(\frac{c_1(c_2-d_2)}{c_1d_2-c_2d_1},\frac{d_2(c_1-d_1)}{c_1d_2-c_2d_1}\right) & \textup{for $i$ even}
\end{cases}.
\end{equation}
\end{theorem}

\begin{proof}
By Proposition~\ref{prop:samesignconverge}, the trajectory sets of points other than $(1,0)$ and $(0, 1)$ have infinitely many points, so there are no rigid cycles.  Trajectories that do not contain $(1,0)$ or $(0,1)$ also have no self-crossing points.   Therefore the denominator $D(\calB^q, \calA_{\bbP}^q)$ can be found by calculating the coordinates of all points on corner trajectory segments of length at most $q$ starting at $(0, 0)$ or $(1, 1)$, and crossing points of augmentations of the same whose lengths sum to at most $q - 1$.  

There are four corner trajectories starting at $(0, 0)$ or $(1,1)$. The trajectory segments $T_1$ and $T_2$ of the form $[\bfb_1, \bfb_2, \ldots, \bfb_q]$ starting at $\bfb_1 = (0, 0)$ with initial velocities $\bfm_1$ and $\bfm_2$ respectively have coordinates
\[ \bfb_i =  \begin{cases}
    \big(1 - \big(\frac{d_1c_2}{c_1d_2}\big)^{\frac{i - 1}{2}}, 0\big) & \textup{for $i$ odd} \\
    \big(1, \frac{d_1}{c_1}\big(\frac{d_1c_2}{c_1d_2}\big)^{\frac{i}{2} - 1}\big) & \textup{for $i$ even}
    \end{cases}\quad
    \textup{ and }\quad
    \bfb_i =  \begin{cases}
    \big(0, 1 - \big(\frac{d_1c_2}{c_1d_2}\big)^{\frac{i - 1}{2}}\big) & \textup{for $i$ odd} \\
    \big(\frac{c_2}{d_2}\big(\frac{d_1c_2}{c_1d_2}\big)^{\frac{i}{2} - 1}, 1\big) & \textup{for $i$ even}
    \end{cases}.
\]
The trajectory segments $T_3$ and $T_4$ of the form $[\bfb_1, \bfb_2, \ldots, \bfb_q]$ starting at $\bfb_1 = (1, 1)$ with initial velocities $\bfm_1$ and $\bfm_2$ respectively have coordinates 
\[ \bfb_i =  \begin{cases}
    \big(\big(\frac{d_1c_2}{c_1d_2}\big)^{\frac{i - 1}{2}}, 1\big) & \textup{for $i$ odd} \\
    \big(0,1 - \frac{d_1}{c_1}\big(\frac{d_1c_2}{c_1d_2}\big)^{\frac{i}{2} - 1}\big) & \textup{for $i$ even}
    \end{cases}\quad
\textup{ and }\quad 
\bfb_i =  \begin{cases}
    \big(1, \big(\frac{d_1c_2}{c_1d_2}\big)^{\frac{i - 1}{2}}\big) & \textup{for $i$ odd} \\
    \big(1 - \frac{c_2}{d_2}\big(\frac{d_1c_2}{c_1d_2}\big)^{\frac{i}{2} - 1}, 0\big) & \textup{for $i$ even}
    \end{cases}.
\]
An example of these trajectory segments is shown in Figure~\ref{fig:oppositesize1}.

\begin{figure}[htp]
\centering
\includegraphics[width=5cm]{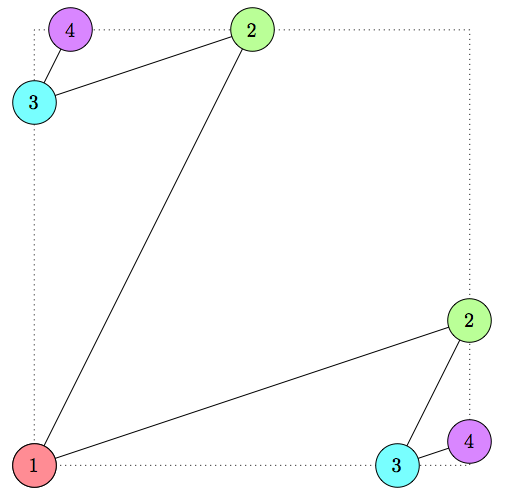}
\includegraphics[width=5cm]{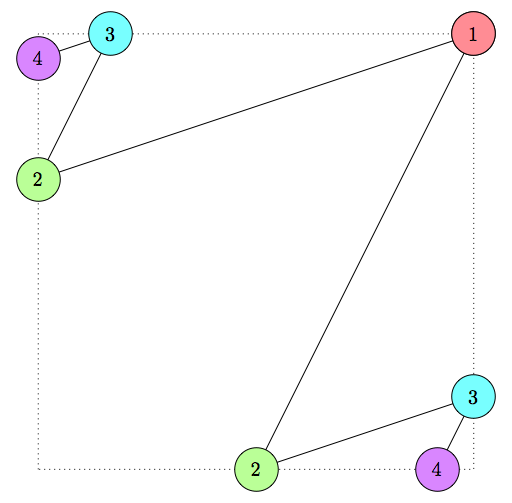}
\includegraphics[width=5cm]{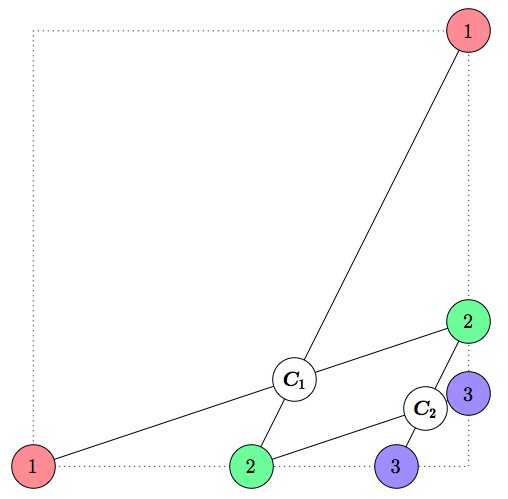}
\caption{For the piece with moves $(1, 2)$ and $(3, 1)$ we illustrate the four corner trajectory segments $T_1$, $T_2$, $T_3$, and $T_4$ starting at $(0,0)$ or $(1,1)$.  The right image shows two crossing points of $T_1$ and $T_4$.}
\label{fig:oppositesize1}
\end{figure}

We must now find all crossing points $\bfp \in \calB^{\circ}$.  We consider crossing points of $T_1$ and $T_4$---the other crossing points arise from a 180-degree rotation around $(\frac{1}{2}, \frac{1}{2})$ and have the same denominators.

Let $T_1 = [\bfa_1, \ldots, \bfa_k]$ and let $T_4 = [\bfb_1, \ldots, \bfb_l]$.  The points lying along $E_1$ starting at $(0,0)$ and moving eastward are $\bfa_1, \bfb_2,\bfa_3,\bfb_4,\ldots$ and the points lying along $E_2$ starting at $(1,1)$ and moving southward are $\bfb_1, \bfa_2,\bfb_3,\bfa_4,\ldots$.  Because line segments only have one of two slopes and because the points are connected in increasing order in the trajectory segment, the only crossing points of line segments from $\bfa_i$ and $\bfa_{i+1}$ and from $\bfb_j$ and $\bfb_{j+1}$ occur when $i=j$.

Solving $\bfp\sim_1\bfa_i$ and $\bfp\sim_2\bfb_i$ for $\bfp$ gives
$$\bfp=\begin{cases}
(1, 0) + \big(\frac{d_1c_2}{c_1d_2}\big)^{\frac{i - 1}{2}}\left(\frac{c_2(d_1 - c_1)}{c_1d_2 - c_2d_1}, \frac{d_1(d_2 - c_2)}{c_1d_2 - c_2d_1}\right) &  \textup{for $i$ odd}\\
(1, 0) + \big(\frac{d_1c_2}{c_1d_2}\big)^{\frac{i}{2}}\left(\frac{c_1(c_2-d_2)}{c_1d_2-c_2d_1},\frac{d_2(c_1-d_1)}{c_1d_2-c_2d_1}\right) & \textup{for $i$ even}
\end{cases},$$
which will be a crossing point when $i \leq \lfloor {(q - 1)}/{2} \rfloor$.  The result follows from Theorem~\ref{thm:denominator}.
\end{proof}

\begin{corollary} \label{cor:inclinednightriderdenominator}
Let $\calB$ be the square board, and let $\bbP$ be the inclined nightrider. Then the denominator of $(\calB^q, \calA_{\bbP}^q)$ is:
\[ \begin{cases}
    1 & q = 1 \\
    2 & q = 2 \\
    3 \cdot 2^{q - 1} & q \geq 3
    \end{cases}.
\]
\end{corollary}
\begin{proof}
For the inclined nightrider with moves $(1,2)$ and $(2,1)$, the denominators in Sequence~\eqref{eq:denominclined1} are $2^{i-1}$ and the denominators in Sequence~\eqref{eq:denominclined2} are $3\cdot 2^{i-1}$, so a factor of 3 will appear in the denominator for all $q\geq 3$. 
\end{proof}

If Conjecture~\ref{conj:period} is true, Corollary~\ref{cor:inclinednightriderdenominator} also provides the formula for the period of the counting quasipolynomial for inclined nightriders.

\subsection{(Some) Slopes of opposite signs}\label{sec:oppsigns}

We now investigate the dynamics of trajectories for a piece $\bbP$ with moves $(c_1, d_1)$ and $(c_2, d_2)$, where the slopes satisfy $0 < {d_1}/{c_1} < 1$, and ${d_2}/{c_2} < -1$.  (This is a generalization of the orthogonal nightrider.)  We let $c_1, d_1, d_2 > 0$ and $c_2 < 0$. In this dynamical system, trajectories converge to a single rigid cycle.  The general case when the moves are of opposite signs is presented as an open question in Section~\ref{sec:properties}.

We first consider real-valued slopes $m_1$ and $m_2$ satisfying $0 < m_1 < 1$ and $m_2 < -1$. The point $\bfb = \big(\frac{m_1 - 1}{m_1 + m_2}, 0\big) \in \partial \calB$ has trajectory set
\begin{equation}\label{eq:calO}
    \calO=\Big\{
    \Big(\frac{m_1 - 1}{m_1 + m_2}, 0\Big), 
    \Big(1, \frac{m_1(1 + m_2)}{m_1 + m_2}\Big), 
    \Big(\frac{1 + m_2}{m_1 + m_2}, 1\Big), 
    \Big(0, \frac{m_2(1 - m_1)}{m_1 + m_2}\Big)
    \Big\}.
\end{equation}
An example is shown in Figure~\ref{fig:convergingrigidcycle}.

The four points of $\calO$ form a rigid cycle because they are the solution to the system of equations 
\[\{\bfz_1\sim_1\bfz_2,~\bfz_2\sim_2\bfz_3,~\bfz_3\sim_1\bfz_4,~\bfz_4\sim_2\bfz_1,~y_1 = 0,~x_2 = 1,~y_3 = 1,~x_4 = 0\},\] which has full rank. In fact, $\calO$ is the only rigid cycle in the system and is an attractor for all other trajectories.

\begin{figure}[tbp] 
\centering
\includegraphics[width=5cm]{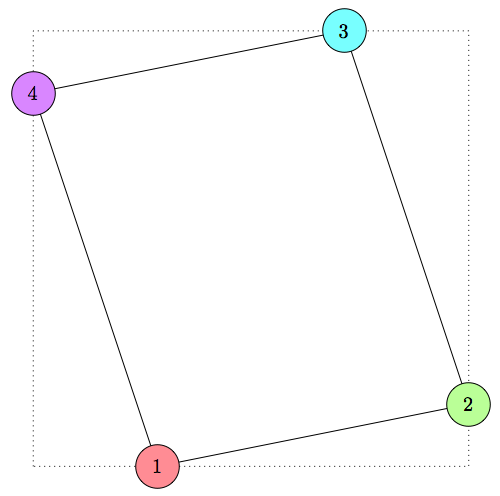}
\includegraphics[width=5cm]{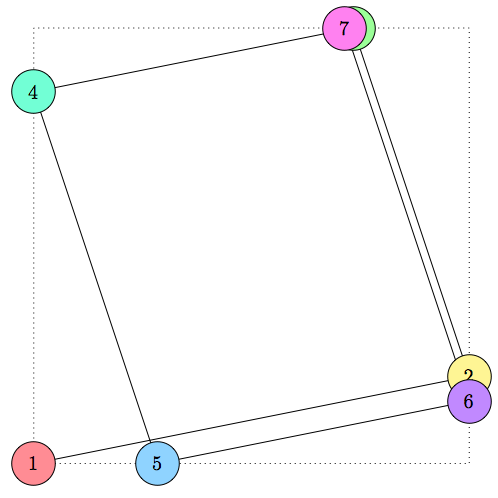}
\caption{The left image shows the rigid cycle $\calO$ in Equation~\eqref{eq:calO} for the piece with moves $(5, 1)$ and $(-1, 3)$.  The right image is a corner trajectory segment in the same system, whose points converge to $\calO$.}\label{fig:convergingrigidcycle}
\end{figure}

\begin{theorem}
\label{thm:orthogonaldynamics}
Let $\calB$ be the square board and let $\bbP$ have moves with real-valued slopes $m_1$ and $m_2$ satisfying $0 < m_1 < 1$ and $m_2 < -1$. The trajectory set $\calO$ in Equation~\eqref{eq:calO} is the only finite trajectory set in $\partial \calB$.  Further, suppose $T=[\bfb_n]_{n\in\bbZ}$ is an trajectory disjoint from $\calO$.  Then as $n$ both increases and decreases, $T$ either stops at a corner or converges to $\calO$. (In other words, $\calO$ is the $\omega$-limit set of $T$.)
\end{theorem}
\begin{proof} 
Restricting the antipode map $s_1$ to the domain $\overline{E_1}$ is a linear contraction $s_1\overline{E_1}\rightarrow \overline{E_2}$ with a factor of $m_1$ because $$|s_1(x_1, 0) - s_1(x_2, 0)| = |(1, m_1(1 - x_1)) - (1, m_1(1 - x_2))| = m_1|x_1 - x_2|.$$ Similarly,  $s_2: \overline{E_2} \rightarrow \overline{E_3}$ is a linear contraction with a factor of $\big|\frac{1}{m_2}\big|$, $s_1: \overline{E_3} \rightarrow \overline{E_4}$ is a linear contraction with a factor of $m_1$, and $s_2: \overline{E_4} \rightarrow \overline{E_1}$ is a linear contraction with a factor of $\big|\frac{1}{m_2}\big|$.  

For any point $\bfb_0\in \partial \calB\setminus\calO$, we investigate the trajectory $T=[\bfb_n]_{n\in\bbZ}$ where we choose $\bfb_1=\phi(\bfb_0)$ to be on the next side counterclockwise from $\bfb_0$.  (This is well defined because of the restrictions on $m_1$ and $m_2$.) By the above reasoning, this sequence continues along sides of $\calB$ in a counterclockwise manner as $n\row+\infty$.  Suppose $o$ is the element of $\calO$ on the same side of $\partial\calB$ as $\bfb_0$.  Then we know that $\phi^4(o)=o$ and \[\big|\phi^{4k}(\bfb_0)-o\big|=\frac{m_1^{2k}}{m_2^{2k}}|\bfb_0-o|.\] We conclude that $\bfb_n$ is defined for all $n\geq 0$ and $\calO$ is the $\omega$-limit set of $T$ as $n\rightarrow\infty$. This also ensures that $\calO$ is the only finite trajectory set.

On the other hand, if we apply $\phi^{-1}$ repeatedly to $\bfb_0$, the points visited can not indefinitely cycle among the sides of $\calB$ in a clockwise manner because each application of $\phi^{-1}$ is an expansion.  Therefore this sequence either stops at a corner, or two successive points $\bfb_{-N+1}$ and $\bfb_{-N}$ are on opposite edges of $\calB$.  When this occurs, $\bfb_{-N-1}$ is on the edge counterclockwise from $\bfb_{-N}$ and the sequence $[\bfb_{-n}]_{n\geq N}$ continues in a counterclockwise manner, which means that it is defined for all $n\geq N$ and $\calO$ is the $\omega$-limit set of $T$ as $n\rightarrow\infty$.
\end{proof}

We now compute $D([0, 1]^{2q}, \calA_{\bbP}^q)$ when $\bbP$ has orthogonal slopes of the form $(m, 1)$ and $(1, -m)$.  

\begin{theorem} \label{thm:orthogonal}
Let $\calB$ be the square board, and let $\bbP$ be the piece with moves $(m, 1)$ and $(1, -m)$. Then the denominator of $(\calB^q, \calA_{\bbP}^q)$ is:
\[ \begin{cases}
    1 & q = 1 \\
    m & q = 2 \\
   m^4+m^2 & q = 3 \\
\lcm(m^2+1, m+1)\cdot m^{q-1} & q \geq 4
    \end{cases}.
\]
\end{theorem}
\begin{proof}
For this piece $\bbP$, the rigid cycle 
\[\calO=\left\{\left({1}/(m+1),0\right),\left(1,{1}/(m+1)\right),\left({m}/(m+1),1\right),\left(0,m/(m+1)\right)\right\}\] contributes a denominator of $m+1$ when $q\geq 4$.  

Each corner is the start of one corner trajectory segment; by symmetry about $(\frac{1}{2},\frac{1}{2})$ the $k$-th point along every trajectory segment has the same denominator.  The trajectory segment $T = [\bfb_1, \bfb_2, \ldots]$ starting at $\bfb_1 = (0, 0)$ has coordinates
\[\bfb_k = \begin{cases}
    \big(0, \frac{m}{m+1}\big) - \frac{1}{m^{k - 1}}\big(0, \frac{1}{m+1}\big) & k \equiv 0 \bmod 4\\
    \big(\frac{1}{m+1}, 0\big) -\frac{1}{m^{k - 1}}\big(\frac{1}{m+1}, 0\big) & k \equiv 1 \bmod 4 \\
    \big(1, \frac{1}{m+1}\big) + \frac{1}{m^{k - 1}}\big(0, \frac{1}{m+1}\big) & k \equiv 2 \bmod 4 \\
    \big(\frac{m}{m+1}, 1\big) + \frac{1}{m^{k - 1}}\big(\frac{1}{m+1}, 0\big) & k \equiv 3 \bmod 4 \\
    \end{cases},
\]
whose denominator is $m^{k-1}$ for all $k$. (Notice, for example, that $m^{k - 1} - 1$ is divisible by $m+1$ for $k$ odd.)

We must also determine the denominators of crossing points of augmentations of trajectory segments and rigid cycles. 
The key insight is that every crossing point $\bfc = (x, y)$ lies on the lines $x - my = r$ and $mx + y = s$ for some rational numbers $r$ and $s$ whose denominators divide the smaller of the denominators of the two points on $\partial\calB$ that the lines intersect.  Solving these equations for $x$ and $y$ we see $x=(r+ms)/(m^2+1)$ and $y=(s-mr)/(m^2+1)$.  In essence, a crossing point of the augmentation of trajectory segments and rigid cycles can not contribute anything new to $([0,1]^{2q}, \calA_{\bbP}^q)$ other than $(m^2+1)$.  This contribution of $(m^2+1)$ will indeed occur when $q\geq 3$ because, for example, the augmentations of the one-point corner trajectory segments $T_a=[(0,0)]$ and $T_b=[(1,0)]$ have the crossing point $\bfc=\big(\frac{m^2}{m^2+1},\frac{m}{m^2+1}\big)$.
\end{proof}

\begin{remark}In the above formula the reader may find it useful to note that
\[\lcm(m^2+1,m+1)=
\begin{cases}
(m^2+1)(m+1) & \textup{if $m$ is even} \\
(m^2+1)(m+1)/2 & \textup{if $m$ is odd} \\
\end{cases}.\]
This is because $\lcm(m^2+1,m+1)=\lcm(m^2-m,m+1)$, and $(m-1)$, $m$, and $(m+1)$ only share a factor if $m$ is odd, for which the common factor is 2.
\end{remark}

The proof for the general case of pieces with orthogonal slopes $(c,d)$ and $(d,-c)$ can be approached similarly but the formula is not nearly as clean.  Theorem~\ref{thm:orthogonal} applies to the orthogonal nightrider with moves $(2, 1)$ and $(1, -2)$.
\begin{corollary}
  \label{cor:orthogonaldenominator}
 Let $\calB$ be the square board, and $\bbP$ be the orthogonal nightrider. Then the denominator of $(\calB^q, \calA_{\bbP}^q)$ is:
\[ \begin{cases}
    1 & q = 1 \\
    2 & q = 2 \\
    20 & q = 3 \\
    15 \cdot 2^{q - 1} & q \geq 4
    \end{cases}
\]
\end{corollary}

As before, this formula would also be the period of the counting quasipolynomial for orthogonal nightriders if Conjecture~\ref{conj:period} is true.

\subsection{Slopes that sum to zero}\label{sec:case3}

We analyze one more case---when the pieces $\bbP$ have moves $(c,d)$ and $(-c,d)$.  In this case, the dynamical system is identical to billiards on a square board.  

A key technique from polygonal billiards is the {\em unfolding} of a trajectory, where the polygon is reflected along edges that the trajectory encounters. (See, for example, Chapter 3 of \cite{tabachnikov}.)  Because the angle of incidence equals the angle of reflection, the trajectory lies along a single line in this unfolded path.  (A visualization is given in  Figure~\ref{fig:unfolding}.)  

\begin{figure}[tbp]
\centering
\includegraphics[width=3.84cm]{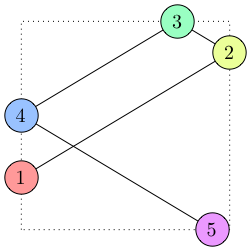}\quad
\includegraphics[width=10.24cm]{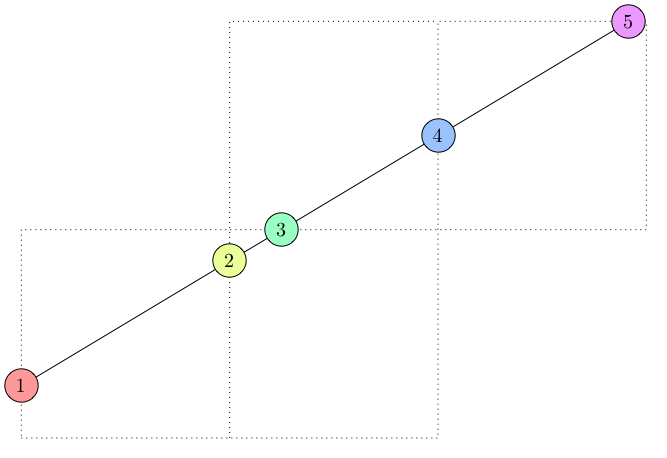}

\caption{The unfolding of a trajectory segment for the piece with moves $(5,3)$ and $(5,-3)$ starting at $\bfb_1=(0, \frac{1}{4})$. The trajectory segment on the left lies on the single line on the right when the unit square is reflected along the edges that are encountered.}
\label{fig:unfolding}
\end{figure}

\begin{proposition} \label{prop:oppsignsrigidcycles}
Let $\calB$ be the square board and let $\bbP$ have moves with with rational slopes $m_1$ and $m_2$ satisfying $m_2 = -m_1$. There are no rigid cycles.
\end{proposition}
\begin{proof}
Section~3.1 of \cite{tabachnikov} shows that on the square board, the trajectory set of every point $\bfb\in\partial\calB$ is finite.  Therefore, trajectory segments that start at a corner must end at a corner and every other trajectory segment is cyclic.

Suppose $T = [\bfb_1, \ldots, \bfb_{l}]$ is a cyclic trajectory segment, with associated hyperplane arrangement $\calH=\calH(\bfb_1, \ldots, \bfb_{l})$. Unfold $T$ starting at $\bfb_1$ along the line $\ell$ defined by $y=m_1x+b$ for some $b\in\calR$.  The integral horizontal and vertical lines ($x=r$ and $y=s$ for integers $r$ and $s$) that the line passes through correspond to the fixations in $\calH$. Because $T$ contains no corners of $\calB$, $\ell$ does not pass through any points in the integer lattice, and therefore there is some $\varepsilon>0$ such that the line $\ell'$ defined by $y=m_1x+b+\varepsilon$ passes through the integral horizontal and vertical lines in the same order and correspond to the same fixations from $\calH$.  We conclude that the trajectory segment $T'$ created by refolding $\ell'$ has the same defining associated hyperplane arrangement as $T$, so $T$ is not a rigid cycle by Lemma~\ref{lem:multiplesolutions}.
\end{proof}

\begin{theorem} \label{thm:oppsignsdenominator}
Let $\calB$ be the square board and let $\bbP$ be the piece with moves $(c, d)$ and $(c, -d)$.  Then $(\calB^q, \calA_{\bbP}^q)$ has denominator
\[ \begin{cases}
    1 & q = 1 \\
    \hat{d} & q = 2 \\
    2\hat{d} & 3 \leq q \leq \big\lceil \hat{d}/\hat{c} \big\rceil \\
    2\hat{c}\hat{d} & q \geq \big\lceil \hat{d}/\hat{c} \big\rceil + 1
    \end{cases},
\]
where $\hat{c}=\min(|c|,|d|)$ and $\hat{d}=\max(|c|,|d|)$.
\end{theorem}
\begin{proof}
By symmetry, we only need to consider the case $0<c<d$. 

Without rigid cycles, the denominator of $(\calB^q, \calA_{\bbP}^q)$ only depends on corner trajectory segments and the crossing points of their augmentations. Let $T = [\bfb_1, \ldots, \bfb_k]$ be the corner trajectory segment starting at $\bfb_1 = (0, 0)$.  Unfold $T$ to lie on the line $\ell$ of slope $d/c$ through $(0,0)$.  For $1\leq i\leq k$, notate the image of $\bfb_i$ under this unfolding to be $\bfb_i'$; observe that $\bfb_i$ and $\bfb_i'$ have the same denominator.  This denominator will either be $c$ or $d$ depending on whether $\ell$ is intersecting a line of the form $x=r$ (for which $\bfb_i'=(r, \frac{dr}{c})$) or a line of the form $y=s$ (for which $\bfb_i'=(\frac{cs}{d},s)$). The denominators of $\bfb_i'$ will all be $d$ until $\ell$ meets the line $x=1$.  Therefore the contribution to the denominator from corner trajectory segments is $1$ if $q=1$, $d$ if $1< q \leq  \lceil d/c \rceil$ and $cd$ when $q>\lceil d/c \rceil$.

We must also determine the relevant crossing points of augmentations of (possibly concurrent) trajectory segments $T_a = [\bfa_1,  \ldots, \bfa_k]$ and $T_b = [\bfb_1,  \ldots, \bfb_l]$. By the above reasoning, every point $\bfa_i$ and $\bfb_i$ is either of the form $(\frac{v_i}{d},u)$ or $(u,\frac{w_i}{c})$ for $u\in\{0,1\}$ and integers $v_i$ and $w_i$, and furthermore because the slopes have magnitude greater than one, at least one endpoint of the line segment between $\bfb_i$ and $\bfb_{i+1}$ (and $\bfb_i$ and $\bfb_{i+1}$) is of the latter form.  This means that any crossing point $\bfc=(x,y)$ can be found by solving two equations of the form
\[y-\frac{w_1}{c}=\frac{d}{c}(x-u_1)\textup{ and }   y-\frac{w_2}{c}=-\frac{d}{c}(x-u_2),\]
from which 
\[x=\frac{du_1+du_2+w_2-w_1}{2d} \textup{ and }y=\frac{du_2-du_1+w_1+w_2}{2c}.\]
Therefore, a crossing point of the augmentation of trajectory segments can not contribute anything to $([0,1]^{2q}, \calA_{\bbP}^q)$ other than $2cd$.  

A contribution of $2$ will definitely occur when $q\geq 3$ because we can see that the augmentations of the one-point corner trajectory segments $T_a=[(0,0)]$ and $T_b=[(0,1)]$ have the crossing point $\bfc=\big(\frac{c}{2d},\frac{1}{2}\big)$.

It remains to show that a contribution of $c$ does not occur when $c>1$ and $q\leq \lceil d/c \rceil$. By symmetry we choose $T_a$ to start at $\bfa_1=(0,0)$ and consider the options for trajectory segments $T_b$ where the lengths of $T_a$ and $T_b$ sum to at most $\lceil d/c \rceil-1$.  If $T_b$ also starts at $(0,0)$, then neither augmented flow reaches $x=1$ and no crossing points exist. If $T_b$ starts at $(0,1)$, again neither augmented flow reaches $x=1$ and the only crossing points are of the form $\bfc=\big(r\frac{c}{2d},\frac{1}{2}\big)$ for odd integers $r$.    If $T_b$ starts at $(0,1)$ or $(1,1)$, the augmentations of $T_a$ does not reach far enough to the right to reach the augmentation of $T_b$. This concludes the proof.
\end{proof}

We now apply Theorem~\ref{thm:oppsignsdenominator} to the lateral nightrider with basic moves $(2,1)$ and $(2,-1)$.

\begin{corollary}
 \label{cor:lateralnightrider}
Let $\calB$ be the square board and $\bbP$ be the lateral nightrider.  Then the denominator of $(\calB^q, \calA_{\bbP}^q)$ is
\[ \begin{cases}
    1 & q = 1 \\
    2 & q = 2 \\
    4 & q \geq 3
    \end{cases}.
\]
\end{corollary}

Again, if Conjecture~\ref{conj:period} is true, this formula would be the period of the counting quasipolynomial for lateral nightriders.

\subsection{Moves that yield periodic trajectories}\label{sec:periodic}
Campbell et al \cite{chaos} investigated the rotation numbers of piecewise linear degree one circle maps.  In investigating the dynamical system in this article (which can be seen as a circle map), Khmelev \cite[p.~558]{khmelev_2005} mentions that it is a difficult question to precisely determine the measure of the set of direction pairs $(\mathbf{m}_1, \mathbf{m}_2)$ that give rational rotation numbers or, equivalently, produce a periodic trajectory or a fixed point.  We are able to conclude that the measure of the set is positive for any polygon.  We prove that this set is of full measure for the triangle and conjecture that this set is not of full measure for any polygon with more than three sides. Some of these results have been found independently by Nogueira and Troubetzkoy. See, for instance, \cite[Corollary 7]{scoopers}. 

For the remainder of this section we consider the probability measure on the space of direction pairs $M = \{(\bfm_1, \bfm_2) \mid \bfm_1, \bfm_2 \in \mathbb{S}^1\}$ that is uniform over the angle parameter and define $N \subseteq M$ be the set of all pairs of nonparallel directions $(\bfm_1, \bfm_2)$ such that $\bfm_1$ and $\bfm_2$ gives rise to a periodic trajectory in the polygon $\calB$ under consideration.  (The restriction of $N$ to nonparallel directions is due to our definition of the dynamical system on two nonparallel directions.)

\begin{theorem}
\label{thm:positive_polygons}
For any convex polygon $\calB$, the measure of $N$ is positive.
\end{theorem}
\begin{proof}
Let $\bfb$ be a corner of $\calB$. Let $\bfv_1$ and $\bfv_2$ be the directions leaving $\bfb$ along the two sides of $\calB$. Define $S\subset \bbS^1$ to be the set of directions that lie in the two (closed) convex cones that are nonnegative linear combinations of $\bfv_1$ and $-\bfv_2$ or of $-\bfv_1$ and $\bfv_2$ in $\bbS^1$.  For any two nonparallel vectors $\bfm_1$ and $\bfm_2$ from $S$, the point $\bfb$ is a fixed point of the corresponding dynamical system. Since the area of these cones is positive and the set of pairs of parallel vectors is of measure 0, the measure of $N$ is positive.
\end{proof}

\begin{proposition}
 \label{prop:triangles}
For any triangle $\calB$, the set $N$ consists of all pairs of nonparallel directions. 
\end{proposition}
\begin{proof}
Let $\calB$ be a triangle $ABC$ whose oriented edges $AB$, $BC$, and $CA$ have direction vectors $\bfv_1$, $\bfv_2$, and $\bfv_3$, respectively. 

Suppose that there exists a pair of nonparallel directions $(\bfm_1,\bfm_2)$ in $N$ that does not create a fixed point at any of $A$, $B$, or $C$. We can conclude that neither of $\bfm_1$ and $\bfm_2$ is parallel to any of $\bfv_1$, $\bfv_2$, or $\bfv_3$ by the following reasoning.  Suppose, for instance, that $\bfm_1$ is parallel to $\bfv_1$. At least one of the two lines with direction vector $\bfm_2$ passing through vertices $A$ and $B$ does not intersect the interior of $\calB$---the corresponding vertex would be a fixed point of the dynamical system.

By the argument in the proof of Theorem~\ref{thm:positive_polygons}, either $\pm\bfm_1$ or $\pm\bfm_2$ must be in each of the (open) convex cones that is the positive or negative linear combination of $\bfv_1$ and $-\bfv_2$, of $\bfv_2$ and $-\bfv_3$, and of $\bfv_3$ and $-\bfv_1$. By the pigeonhole principle, either $\pm\bfm_1$ or $\pm\bfm_2$ is in at least two of these cones.  However, because $\calB$ is a triangle, these sets are all disjoint, so no such $\bfm_1$ and $\bfm_2$ exist. 
\end{proof}

\begin{theorem} \label{thm:positive_measure}
For the square $\calB = [0, 1]^2$, the measure of $N$ is at least $3/4$.
\end{theorem}
\begin{proof}
Write $\bfm_1 = (c_1, d_1)$ and $\bfm_2 = (c_2, d_2)$. By the discussion preceding Theorem \ref{thm:orthogonaldynamics}, if $0 < \frac{d_1}{c_1} < 1$ and $\frac{d_2}{c_2} < -1$, then there is a periodic orbit in $\calB$. The set of such direction pairs is a set of measure $1/4$. Moreover, direction pairs $\bfm_1$ and $\bfm_2$ with $\frac{d_1}{c_1}$ and $\frac{d_2}{c_2}$ having the same signs give rise to fixed points as in Section~\ref{sec:case1}. The set of such direction pairs has measure $1/2$. As before, the set of pairs of parallel vectors is of measure 0, so $N$ has measure at least $3/4$.
\end{proof}

A natural question is whether $N$ has full measure. We believe the answer is no for all polygons with four or more sides.

\begin{conjecture} \label{conj:measure_polygons}
For all polygons with four or more sides, the measure of $N$ is strictly less than $1$.
\end{conjecture}

\begin{conjecture} \label{conj:measure_squares}
For the square $\calB = [0, 1]^2$, the measure of $N$ is $3/4$.
\end{conjecture}

Our intuition for these conjectures comes from behavior of the dynamical system on the square for choices of direction slopes $m_1$ and $m_2$ not covered in earlier sections of Section~\ref{sec:applications}. In particular, the case when $m_1$ and $m_2$ have opposite signs is not fully understood. 

Several types of dynamics have emerged in this case.  The simplest situation is when all trajectory segments are cyclic.  This occurs when $m_2 = -m_1$ (see Section~\ref{sec:case3}) and this also appears to occur when $m_1 = \frac{1}{3}$ and $m_2 = -\frac{2}{3}$. (See Figure~\ref{fig:AllPeriodicOpenQuestion}.)

\begin{figure}[htp]
\centering
\includegraphics[width=5cm]{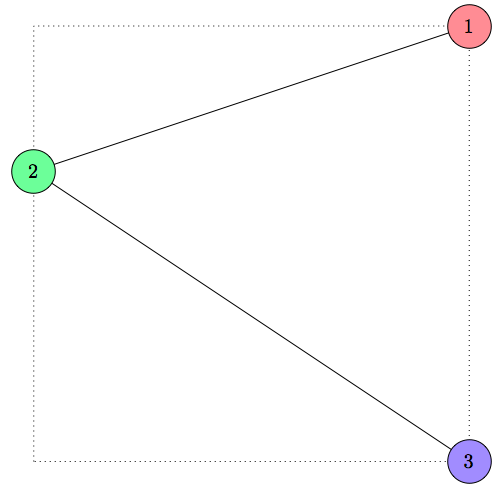}
\includegraphics[width=5cm]{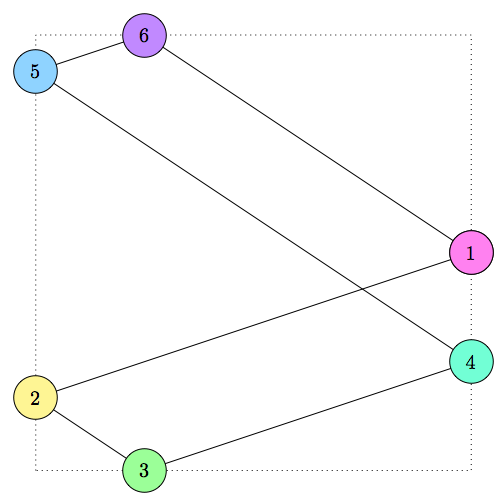}
\caption{$m_1 = \frac{1}{3}$ and $m_2 = -\frac{2}{3}$. The first trajectory segment begins at $(1, 0)$, while the second begins at $(1, \frac{1}{2})$. The other points we tested on $\partial\calB$ also have periodic orbits.}
\label{fig:AllPeriodicOpenQuestion}
\end{figure}

Convergent behavior also occurs, similar to what we saw in Figure~\ref{fig:convergingrigidcycle} from Section~\ref{sec:oppsigns} in which all trajectories converge to the same rigid cycle.  When $m_1 = \frac{3}{10}$ and $m_2 = -\frac{4}{10}$, trajectories converge to a single finite trajectory set, as shown in Figure~\ref{fig:ConvergingOpenQuestion}.

\begin{figure}[htp]
\centering
\includegraphics[width=5cm]{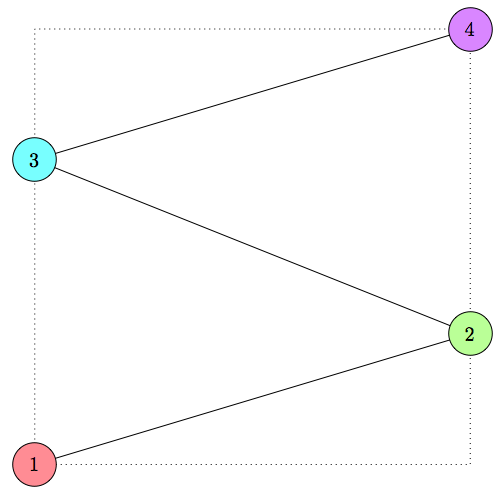}
\includegraphics[width=5cm]{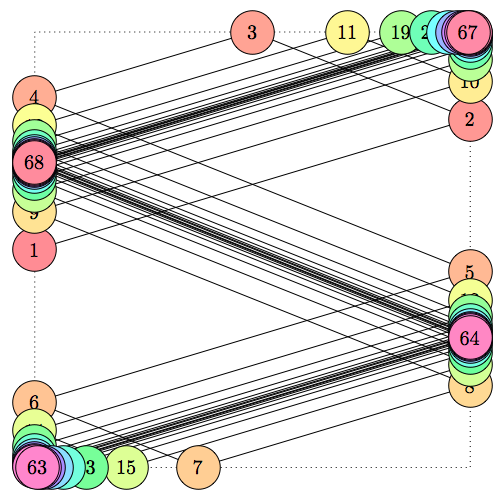}
\caption{$m_1 = \frac{3}{10}$ and $m_2 = -\frac{4}{10}$. The first trajectory begins at $(0, 0)$, and the second begins at $(0, \frac{1}{2})$. The points of the first trajectory seem to form the $\omega$-limit set of the second.}
\label{fig:ConvergingOpenQuestion}
\end{figure}

However, these two behaviors producing finite trajectory sets seem to hinge on relationships between the values of $m_1$ and $m_2$.  Much more common is an ergodic behavior.  For example, when $m_1 = \frac{1}{3}$ and $m_2 = -\frac{1}{4}$ we have the behavior shown in Figure~\ref{fig:DenseOpenQuestion}.  This orbit does not seem to converge to any limiting trajectory, unlike the convergence behavior seen when move combinations give rise to periodic orbits (as in Figures \ref{fig:convergingrigidcycle} and \ref{fig:ConvergingOpenQuestion}).

\begin{figure}[htp]
\centering
\includegraphics[width=5cm]{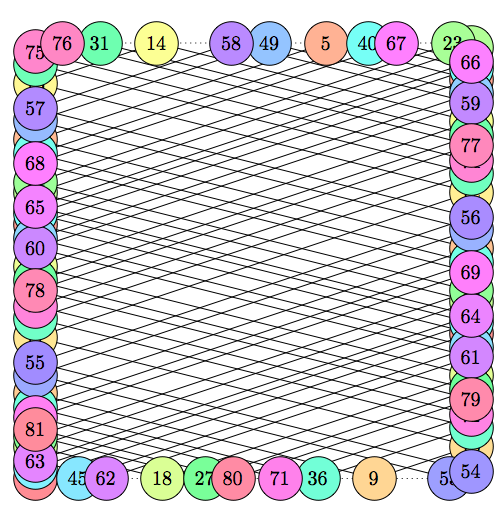}
\caption{$m_1 = \frac{1}{3}$ and $m_2 = -\frac{1}{4}$. These are the first 80 points in the trajectory set of $(0, 0)$, which appears to be dense in $\partial\calB$.
}
\label{fig:DenseOpenQuestion}
\end{figure}

Further substantiating ergodic behavior is that the trajectory set of $(0, 0)$ appears to be dense in $\partial\calB$.  We collected data about the positions of all trajectory points that lie on the western edge of the square in Figure~\ref{fig:DenseOpenQuestion} and we computed the maximum distance gap between any two trajectory points, in other words, the length of the largest segment that had no trajectory points. This gap distance decreases steadily as more and more points are computed.  See Table~\ref{tab:dense_numerical_data}.

\begin{table}[htp]
    \centering
\begin{tabular}{c|cccccccccc}
$N$ &  3 & 8 & 11 & 16 & 21 & 29 & 52 & 73 & 
96  \\ \hline 
gap &  0.583 & 0.389 & 0.259 & 0.194 & 0.162 & 0.130 & 0.0864 & 0.0576 & 0.0432 \vspace{.1in} \\ 
$N$ & 119& 179 & 308 & 435 & 564 & 693 & 1053 & 1800 & 2545  \\ \hline
gap   & 0.0360 & 0.0288 &  0.0192 & 0.0128 & 0.00960 & 0.00800 & 0.00640 & 0.00427 & 0.00285  \\
\end{tabular}
\medskip

   \caption{The largest distance gap left between any two points in the finite trajectory segment on the western edge of the square in  Figure \ref{fig:DenseOpenQuestion} as a function of the number of points plotted. The given $N$ is the first number of points for which the gap decreased to this value.}
    \label{tab:dense_numerical_data}
\end{table}

If the points visited by the trajectory are indeed dense in the boundary, Theorem~\ref{thm:treacheryinterval} may be useful to prove Conjecture~\ref{conj:measure_squares}---no treachery could exist because there would be no open intervals available around its boundary points.

\section{Open Questions}\label{sec:open}

The variables that determine the behavior of a particle's flow in mathematical billiards are the shape of the region as well as the initial position and initial direction of the particle.  In the dynamical system studied in this article, the key variables are the shape of the board, the slopes of the moves, and the initial position of the particle.  The similarity of the behaviors of the flows in the two dynamical systems leads to many open questions.  Progress on Questions~7.1, 7.2, 7.3, 7.4, and 7.13 has appeared in Nogueira and Troubetzkoy's \cite{scoopers}.

\subsection{Fruitful regions and moves}
In the study of convex billiards, circles, ellipses, and curves of constant width have produced beautiful results.  So have rational polygons, where internal angles are rational multiples of~$\pi$ \cite{tabachnikov}.  This leads us to ask which choices of board and moves lead to fruitful directions in the dynamical system of this article.

\begin{question}
What properties of polygonal or general convex boards are more likely to lead to provable results, in terms of dynamical properties or Ehrhart theory, for some choices of moves?
\end{question}  

\begin{question}
What restrictions on moves are more likely to lead to provable results, in terms of dynamical properties or Ehrhart theory, on a wide variety of boards?
\end{question}

\subsection{Properties of trajectories}\label{sec:properties}

In convex billiards, a classic unsolved question is whether every polygon has a periodic orbit, which has applications to the physics of point masses \cite{gutkin2}. It is known that every rational polygon and every acute triangle has a periodic orbit.  For square regions, it is further known that a billiard trajectory is periodic if the slope of the particle's initial direction is rational, and ergodic otherwise.  We ask similar questions about the dynamical system in this article. 

\begin{question}\label{q:periodic}
Given a polygonal board $\calB$ (or an arbitrary convex board $\calB$), what conditions on the slopes $m_1$ and $m_2$ will ensure that there is a periodic orbit in $\calB$?
\end{question}

\begin{question}\label{q:ergodic}
For which choice of board $\calB$, slopes $m_1$ and $m_2$, and initial point $\bfb$ is the trajectory through $\bfb$ ergodic?
\end{question}

To apply Theorems~\ref{thm:vertices} and \ref{thm:denominator}, we must understand the periodic orbits and also be able to determine the rank of their corresponding cyclic trajectory segments.  This leads to the following refinement of the Question~\ref{q:periodic}.

\begin{question}\label{q:rigid}
For which choice of board $\calB$ and slopes $m_1$ and $m_2$ does there exist a rigid cycle?  And under which conditions is there a unique rigid cycle?
\end{question}

The variety of behaviors for pieces with slopes of opposite signs in Section~\ref{sec:periodic} leads us to ask for a classification for these behaviors on the square board.

\begin{question}\label{q:square}
Classify the behavior of trajectories on the square board for every choice of pieces with moves along slopes $m_1$ and $m_2$.  Under what conditions will there be a periodic orbit and what is it?  Under what conditions will the behavior of the system be ergodic?  
\end{question}

We remark that in Sections~\ref{sec:case1} and \ref{sec:oppsigns} the dynamics do not depend on the rationality of $m_1$ and $m_2$, but in Section~\ref{sec:case3} they do.  We are not sure why this is the case.

\begin{question}
Which results hold for irrational slopes in addition to rational slopes?
\end{question}

\subsection{Generalizations of the dynamical system} 
There are many ways that the discrete dynamical system for billiards generalizes; we wonder if this other model can also be generalized further.  First, we ask if it is possible to generalize the board $\calB$ to regions that are fruitful in the study of billiards.

\begin{question}
Can the dynamical system in this article be generalized to non-convex regions? To hyperbolic models?  To a system similar to outer billiards?
\end{question}

We also wonder if we can remove the restriction that there are only two moves. 

\begin{question}\label{q:moremoves}
Is there a way to make sense of such a dynamical system involving more than two moves? 
\end{question}

Could studying such a dynamical system be useful in the study of three-move riders, or riders with more moves?  One possible way to allow for more moves is to require that the moves be applied in a cyclical fashion.  When there are only two moves, the trajectory must always lie in the plane spanned by those two vectors. If one is able to find a way to involve more than two moves, the dynamical system may be able to generalize to higher dimensions.

\begin{question}
Is there a higher-dimensional analog of this dynamical system, similar to billiards in a polytope?
\end{question}

\subsection{Dynamical System Theory}

Inspired by dynamical systems theory we can ask about the stability of the dynamical system by perturbing the board, perturbing the set of moves, and perturbing the particle's initial position.

\begin{question}
How does a slight perturbation of the board impact the behavior of the trajectories?  Of the existence or uniqueness of the rigid cycles?  How do the changes depend on the piece's moves?
\end{question}

\begin{question}
How does a slight perturbation of the piece's move vectors impact the behavior of the trajectories?  Of the existence or uniqueness of the rigid cycles?  How do the changes depend on the board?
\end{question}

\begin{question}\label{q:close}
Do two trajectories that start from sufficiently close points $\bfb$ and $\bfb'$ have the same behavior?  If $\bfb$ is periodic, must $\bfb'$ be periodic?  Must they have the same period?
\end{question}

Question~\ref{q:close} was partially answered in Theorem~\ref{thm:treacheryinterval} when the trajectory is a treachery.  In general, if the answer to Question~\ref{q:close} is positive for a specific board and set of moves, that would prove that the corresponding cyclic trajectory segments are not rigid cycles, similar to the argument given in Proposition~\ref{prop:oppsignsrigidcycles}.

Crossing points of trajectories are central to the study of the dynamical system, but there does not appear to be much focus on them in the discrete dynamical system literature. Perhaps such a question can inspire new directions of research in existing dynamical systems.

\begin{question}
What are the coordinates of crossing points of trajectories in existing discrete dynamical systems, including billiards?  For which discrete dynamical systems are the formulas of the coordinates of these crossing points easy to calculate?  Do the denominators of these coordinates behave predictably?
\end{question}

\subsection{Periods and Denominators}
An important question in Ehrhart Theory is the relationship between the period of an Ehrhart quasipolynomial and the denominator of its corresponding polytope (or inside-out polytope). 

We have found the denominator of $(\calB^q, \calA_{\bbP}^q)$ for several classes of two-move riders $\bbP$ when $\calB$ is the square board. This gives us provable bounds on the period of the Ehrhart quasipolynomial of $(\calB^q, \calA_{\bbP}^q)$, and we can use this to explicitly compute $u_{\bbP}(q; n)$ through brute force. This may give insight on the period of $u_{\bbP}(q; n)$. 

\begin{question}
Is the period always equal to the denominator of $(\calB^q, \calA_{\bbP}^q)$ when $\bbP$ is a two-move rider?
\end{question}


\section*{Acknowledgments}

We would like to thank Thomas Zaslavsky and Dan Lee for fruitful discussions.  We are grateful to multiple anonymous referees who have pointed us to the origin of the dynamical system, other works where it has appeared, and guidance for greatly improving our dynamical systems exposition.  The first author is grateful for the support of PSC-CUNY Award 61049-0049. 

\bibliographystyle{plain}

\end{document}